\documentclass[11pt]{article}
\usepackage{amsthm,amstext, amsmath,latexsym,amsbsy,amssymb,cases}

\newtheorem{cor}{Corollary}[section]

\newtheorem{cl}[cor]{Claim}
\newtheorem{lem}[cor]{Lemma}
\newtheorem{prop}[cor]{Proposition}

\newtheorem{theorr}{Theorem}
\newtheorem{propp}[theorr]{Proposition}

\title{\bf A modulation technique for the blow-up profile of the vector-valued semilinear wave equation}
 \author{\small Asma AZAIEZ\footnote{This author is supported by the ERC Advanced Grant no. 291214, BLOWDISOL.}  \\
\small Universit\'e de Cergy-Pontoise, \\
\small AGM, CNRS (UMR 8088), 95302, Cergy-Pontoise, France.\\
\small Hatem ZAAG\footnote{This author is supported by the ERC Advanced Grant no. 291214, BLOWDISOL. and by ANR project ANA\'E ref. ANR-13-BS01-0010-03.} \\
\small Universit\'e Paris 13, Sorbonne Paris Cit\'e \\
\small LAGA, CNRS (UMR 7539), F-93430, Villetaneuse, France.}
\begin{document}
\maketitle
\begin{abstract} We consider a vector-valued  blow-up solution with values in $\mathbb{R}^m$ for the semilinear wave equation with power nonlinearity in one space dimension (this is a system of PDEs). We first characterize all the solutions of the associated stationary problem as an m-parameter family. Then, we show that the solution in self-similar variables approaches some particular stationary one in the energy norm, in the non-characteristic cases. Our analysis is not just a simple adaptation of the already handled real or complex case. In particular, there is a new structure of the set a stationary solutions.
 \end{abstract}
vector valued PDE.

\begin{flushright}
{\it Paper accepted for publication in Bull. Sci. math.\;\;\;\;\;\;}
\end{flushright}
\section{Introduction}

We consider the vector-valued semilinear wave equation
\begin{equation} \left\{
\begin{array}{l}
\displaystyle\partial^2_{t} u = \partial^2_{x} u+|u|^{p-1}u, \\
u(0)=u_{0} \mbox{ and }  u_{t}(0) = u_{1},
\end{array}
\right . \label{waveq}
\end{equation}
where here and all over the paper $|.|$ is the euclidian norm in $\mathbb{ R}^m$, $u(t):  x\in \mathbb{R}\to u(x,t) \in \mathbb{R}^m,\, m\ge 2,\,p>1$, $u_0 \in
H^1_{loc,u}$ and $ u_1\in L^2_{loc,u}$ with $||
v||^2_{L^2_{loc,u}}=\displaystyle\sup\limits_{a\in \mathbb{ R}}
\int_{|x-a|<1}|v(x)|^2 dx $ and $|| v||^2_{H^1_{loc,u}}=||
v||^2_{L^2_{loc,u}}+||  \nabla v||^2_{L^2_{loc,u}}\cdot$

\medskip
The Cauchy problem for equation (\ref{waveq}) in the space $H^1_{loc,u}\times L^2_{loc,u}$ follows from the finite speed of propagation and the wellposedness in $H^1\times L^2$. See for instance Ginibre, Soffer and Velo \cite{MR1190421}, Ginibre and Velo \cite{MR1256167}, Lindblad and Sogge \cite{MR1335386} (for the local in time wellposedness in $H^1\times L^2$). Existence of blow-up solutions follows from ODE techniques or the energy-based blow-up criterion of \cite{l74}.
 More blow-up results can be found in Caffarelli and Friedman \cite{MR849476}, Alinhac \cite{ali95} and \cite{MR1968197}, Kichenassamy and Littman \cite{KL93a}, \cite{KL93b} Shatah and Struwe \cite{MR1674843}). 
 \medskip
 
 The real case (in one space dimension) has been understood completely, in a series of papers by Merle and Zaag \cite{MR2362418}, \cite{MR2415473}, \cite{MR2931219} and \cite{Mz12} and in C\^ote and Zaag \cite{CZ12} (see also the note \cite{Mz10}). Recently, the authors give an  extension to higher dimensions in \cite{MZ13} and \cite{MZ14}, where the blow-up behavior is given, together with some stability results.

For other types of nonlinearities, we mention the recent contribution of Azaiez, Masmoudi and Zaag in \cite{3}, where we study the semilinear wave equation with exponential nonlinearity, in particular we give the blow-up rate with some estimations.

In \cite{<3}, we consider the complex-valued solution of (\ref{waveq}) (or $\mathbb{R}^2$-valued solution), characterize all stationary solutions and give a trapping result. The main obstruction in extending those results to the vector case $m\ge 3$ was the question of classification of all self similar solutions of (\ref{waveq}) in the energy space. In this paper we solve that problem and show that the real valued and complex valued classification also hold in the vector-valued case $m\ge 3$ (see Proposition \ref{p1} below), with an adequate choice in $S^{m-1}$. This is in fact our main contribution in this paper, and it allows us to generalize the results of the complex case to the vector valued case $m\ge 3$. In this paper, we aim at proving similar results for the general case $u(x,t) \in \mathbb{R}^m,$ for $m\ge 3$.

Let us first introduce some notations before stating our results.

\medskip
If \emph{u} is a blow-up solution of (\ref{waveq}), we define (see
for example Alinhac \cite{ali95}) a continuous curve $\Gamma$ as the graph of a function ${x \rightarrow T(x)}$ such that the domain of definition of $u$ (or the maximal influence domain of $u$) is 
\begin{equation*}\label{domaine-de-definition}
D_u=\{(x,t)| t<T(x)\}.
\end{equation*}

From the finite speed of propagation, $T$ is a 1-Lipschitz function. The time $\bar {T}=\inf_{x \in \Bbb R}T(x)$ and the graph $\Gamma$  are called (respectively) the blow-up time and the blow-up graph of $u$.

Let us introduce the following non-degeneracy condition for $\Gamma$. If we introduce for all $x \in \Bbb{R},$ $t\le T(x)$ and $\delta>0$, the cone
\begin{equation*}
 \mathcal{C}_{x,t, \delta }=\{(\xi,\tau)\neq (x,t)\,| 0\le \tau\le t-\delta |\xi-x|\},
\end{equation*}
then our non-degeneracy condition is the following: $x_0$ is a non-characteristic point if
\begin{equation}\label{4}
 \exists \delta = \delta (x_0) \in (0,1) \mbox{ such that } u \mbox{ is defined on }\mathcal{C}_{x_0,T(x_0), \delta_0}.
\end{equation}
If condition (\ref{4}) is not true, then we call $x_0$ a characteristic point. Already when $u$ is real-valued, we know from \cite{MR2931219} and \cite{CZ12} that there exist blow-up solutions with characteristic points.

\medskip

Given some $x_0 \in \Bbb R,$ we introduce the following self-similar change of variables:
\begin{equation}\label{trans_auto}
w_{x_0}(y,s) =(T(x_0)-t)^\frac{2}{p-1}u(x,t), \quad  y=\frac{x-x_0}{T(x_0)-t}, \quad
s=-\log(T(x_0)-t).
\end{equation}
This change of variables transforms the backward light cone with vertex $(x_0, T(x_0))$ into the infinite cylinder $(y,s)\in (-1,1) \times [-\log T(x_0),+\infty).$ The function $w_{x_0}$ (we write $w$ for simplicity) satisfies the following equation for all $|y|<1$ and $s\ge -\log T(x_0)$:
\begin{eqnarray} \label{equa}
 \partial^2_{s} w=\mathcal{L}w-\frac{2(p+1)}{(p-1)^2}w+|w|^{p-1}w-\frac{p+3}{p-1} \partial_s w- 2 y \partial_{ys} w
 \end{eqnarray}
 \begin{eqnarray}\mbox{where }\mathcal{L} w=\frac{1}{\rho}\partial_y (\rho (1-y^2)\partial_y w)\, \mbox{ and }\, \rho (y)= (1-y^2)^\frac{2}{p-1}.\label{8}
\end{eqnarray}
This equation will be studied in the space
 \begin{eqnarray}\label{9}
\mathcal{ H}=\{ q \in H_{loc}^1 \times L_{loc}^2 ((-1,1),\mathbb{R}^m) \, \Big| \parallel  q \parallel_{\mathcal{ H}}^2 \equiv \int_{-1}^{1}(|q_1|^2+|q'_1|^2(1-y^2)+|q_2|^2)\rho\;dy< +\infty \} ,
 \end{eqnarray}
which is the energy space for $w$. Note that $\mathcal{ H}=\mathcal{ H}_0\times L_{\rho}^2$ where
 \begin{eqnarray}\label{10}
\mathcal{ H}_0=\{ r \in H_{loc}^1  ((-1,1),\mathbb{R}^m) \,\Big| \parallel  r \parallel_{\mathcal{ H}_0}^2 \equiv \int_{-1}^{1}(|r'|^2(1-y^2)+|r|^2)\rho\;dy< +\infty\} .
 \end{eqnarray}
In some places in our proof and when this is natural, the notation $\mathcal{ H}$, $\mathcal{ H}_0$ and $L_{\rho}^2$ may stand for real-valued spaces.
Let us define 
\begin{equation}\label{15}
 E(w,\partial_s w)=\int_{-1}^{1} \left( \frac{1}{2} |\partial_s w|^2+\frac{1}{2} |\partial_y w|^2 (1-y^2)+\frac{p+1}{(p-1)^2}|w|^2-\frac{1}{p+1}|w|^{p+1}\right) \rho dy.
\end{equation}
By the argument of Antonini and Merle \cite{AM01}, which works straightforwardly in the vector-valued case, we see that $E$ is a Lyapunov functional for equation (\ref{equa}).

\subsection{Blow-up rate}

Only in this subsection, the space dimension will be extended to any $N \ge 1$. We assume in addition that $p$ is conformal or sub-conformal:
\begin{equation*}
 1<p\le p_c \equiv 1+\frac{4}{N-1}.
\end{equation*}
We recall that for the real case of equation (\ref{waveq}), Merle and Zaag determined in \cite{MZ05} and \cite{MZ2005} the blow-up rate for (\ref{waveq}) in the region $\{(x,t)\,|\, t<\bar T\}$ in a first step. Then in \cite{MR2147056}, they extended their result to the whole domain of definition $\{(x,t)\,|\, t< T(x)\}$. In fact, the proof of \cite{MZ05}, \cite{MZ2005} and \cite{MR2147056} is valid for vector-valued solutions, since the energy structure (see(\ref{15})), which is the main ingredient of the proof, is preserved. This is the growth estimate near the blow-up surface for solutions of equation (\ref{waveq}).
\begin{propp}\label{Th}{\bf (Growth estimate near the blow-up surface for solutions of equation (\ref{waveq}))}
If $u$ is a solution of (\ref{waveq}) with blow-up surface $\Gamma \,:\, \{x\rightarrow T(x)\},$ and if $x_0\in \mathbb{R}^N$ is non-characteristic (in the sense (\ref{4})) then,\\
\item{(i)} {\bf (Uniform bounds on $w$)} For all $s\ge -\log \frac{T(x_0)}{4}$:
$$||w_{x_0}(s)||_{H^1(B)}+||\partial_s w_{x_0}(s)||_{L^2(B)}\le K.$$
\item{(ii)} {\bf (Uniform bounds on $u$)} For all $t\in [\frac{3}{4}T(x_0),T(x_0))$:
\begin{align*}
 &(T(x_0)-t)^\frac{2}{p-1}\frac{||u(t)||_{L^2(B(x_0,T(x_0)-t))}}{T(x_0)-t)}\\
&+(T(x_0)-t)^{\frac{2}{p-1}+1}\left( \frac{||\partial_t u(t)||_{L^2(B(x_0,T(x_0)-t))}}{(T(x_0)-t)^{N/2}}+\frac{||\nabla u(t)||_{L^2(B(x_0,T(x_0)-t))}}{(T(x_0)-t)^{N/2}}\right) \le K,
\end{align*}
where the constant $K$ depends only on $N,\, p,$ and on an upper bound on $T(x_0),1/T(x_0)$, $\delta_0 (x_0)$ and the initial data in $H^1_{loc,u}\times L^2_{loc, u}$.
\end{propp}

\medskip
\subsection{Blow-up profile}
This result is our main novelty. In the following, we characterize the set of stationary solutions for vector-valued solutions.
\begin{propp}\label{p1}{\bf (Characterization of all stationary solutions of equation (\ref{equa}) in  $\mathcal{ H}_0$).}
 {\it(i) Consider $w \in \mathcal{H}_0$ a stationary solution of (\ref{equa}). Then, either $w\equiv 0$ or there exist $d\in(-1,1)$ and $\Omega \in \mathbb{S}^{m-1}$ such that $w(y)=\Omega \kappa (d,y)$ where
\begin{eqnarray}\label{defk}
\forall (d, y) \in (-1,1)^2,\;\kappa(d,y)= \kappa_0 \frac{(1-d^2)^\frac{1}{p-1}}{(1+dy)^\frac{2}{p-1}} \mbox{ and }\kappa_0=\left(\frac{2(p+1)}{(p-1)^2}\right)^\frac{1}{p-1}.
\end{eqnarray}}\\
(ii) It holds that
\begin{equation}\label{14}
 E(0,0)=0 \mbox{ and }\forall d \in (-1,1),\,\forall \Omega \in \mathbb{S}^{m-1} ,\, E( \kappa (d,.)\Omega,0)=E(\kappa_0,0)>0
\end{equation}
where $E$ is given by (\ref{15}).
\end{propp}

Thanks to the existence of the Lyapunov functional $E(w,\partial_s w)$ defined in (\ref{15}), we show that when $x_0$ is non-characteristic, then $w_{x_0}$ approaches the set of non-zero stationary solutions:
\begin{propp}\label{2}{\bf (Approaching the set of non-zero stationary solutions near a non-characteristic point)}
 Consider $u$ a solution of (\ref{waveq}) with blow-up curve $\Gamma :\{ x \rightarrow T(x)\}.$ If $x_0 \in \Bbb R$ is non-characteristic, then:
\item{(A.i)} $\inf_{\{\Omega \in {\Bbb S}^{m-1},\; |d|<1\}}||w_{x_0}(.,s)- \kappa(d,.)\Omega||_{H^1(-1,1)}+||\partial_s w_{x_0}||_{L^2(-1,1)} \rightarrow  0$ as $s  \rightarrow \infty$.
\item{(A.ii)} $E(w_{x_0}(s),\partial_s w_{x_0}(s))\rightarrow E(\kappa_0,0)$ as $s \rightarrow \infty.$
\end{propp}
We write the fundamental theorem of our paper:

\medskip
\begin{theorr}\label{theo3}{\bf (Trapping near the set of non-zero stationary solutions of (\ref{equa}))} 
There exist positive $\epsilon_0$, $\mu_0$ and $C_0$ such that if $w\in C([s^*,\infty),\mathcal{H})$ for some $s^*\in \mathbb{R}$ is a solution of equation (\ref{equa}) such that
 \begin{eqnarray} \label{17}
  \forall s \ge s^*, E(w(s),\partial_s w(s)) \ge E(\kappa_0,0),
 \end{eqnarray}
and
 \begin{eqnarray}\label{18}
\Big|\Big|\begin{pmatrix} w(s^*)\\\partial_s w(s^*) \end{pmatrix} -\begin{pmatrix} \kappa(d^*,.)\Omega^*\\0\end{pmatrix} \Big|\Big|_{\mathcal{ H}}\le \epsilon^*
 \end{eqnarray}
for some $d^* \in (-1,1), \Omega^* \in \mathbb{S}^{m-1}$ and $\epsilon^*\in (0,\epsilon_0]$, 
then there exists $d_{\infty} \in (-1,1)$ and $\Omega^\infty\in \mathbb{S}^{m-1}$ such that
$$|\arg\tanh d_{\infty}-\arg \tanh d^*|+|\Omega_\infty-\Omega^*| \le C_0 \epsilon^* $$
and for all $s\ge s^*$:
 \begin{eqnarray}\label{19}
\Big|\Big|\begin{pmatrix} w(s)\\\partial_s w(s) \end{pmatrix} -\begin{pmatrix} \kappa(d_\infty,.)\Omega_{\infty}\\0\end{pmatrix} \Big|\Big|_{\mathcal{ H}}\le C_0\epsilon^* e^{-\mu_0(s-s^*)}.
 \end{eqnarray}
\end{theorr}
Combining Proposition \ref{2} and Theorem \ref{theo3}, we derive the existence  of a blow-up profile near non-characteristic points in the following:

\begin{theorr}\label{theo4}{\bf (Blow-up profile near a non-characteristic point)}
 If $u$ a solution of (\ref{waveq}) with blow-up curve $\Gamma :\{ x \rightarrow T(x)\}$ and $x_0 \in \Bbb R$ is non-characteristic (in the sense (\ref{4})), then there exist $d_\infty (x_0) \in (-1,1),\, \Omega^\infty(x_0)\in \mathbb{S}^{m-1}$ and $s^*(x_0) \ge -\log T(x_0)$ such that for all $s\ge s^*(x_0),$ (\ref{19}) holds with $\epsilon^*=\epsilon_0$, where $C_0$ and $\epsilon_0$ are given in Theorem \ref{theo3}. Moreover,
 \begin{equation*}
||w_{x_0}(s)- \kappa(d_\infty (x_0))\Omega^\infty(x_0)||_{H^1(-1,1)}+||\partial_s w_{x_0}(s)||_{L^2(-1,1)} \rightarrow  0 \mbox{ as } s  \rightarrow \infty.
  \end{equation*}
\end{theorr}
\noindent{\bf Remark:} From the Sobolev embedding, we know that the convergence takes place also in $L^\infty$, in the sense that
$$||w_{x_0}(s)- \kappa(d_\infty (x_0))\Omega^\infty(x_0)||_{L^\infty(-1,1)} \rightarrow  0 \mbox{ as } s  \rightarrow \infty.$$

\bigskip

In this paper, we give the proofs of Proposition \ref{p1} and Theorem \ref{theo3}, which present the novelties of this work comparing with the handled real and complex cases,  since Propositions \ref{Th}, \ref{2} and Theorem \ref{theo4}, can be generalized from the real case treated in \cite{MR2362418} without any difficulty.\\

Let us remark that our paper is not a simple adaptation of the complex case. In fact, the vector-valued structure of our solution implies a new characterization of the set of stationary solutions in $\Bbb{R}^m$ (see Proposition \ref{p1} above).  In addition, in order to apply the modulation theory, we need more parameters, and for that, a suitable $ m\times m$ rotation matrix will be defined (see (\ref{matrice}) and (\ref{R_theta}) below; see the beginning of the proof of Proposition \ref{5.2} page \pageref{chdg} below), and we have to treat delicately the terms coming from the rotation matrix.\\

 This paper is organized as follows:\\
- In Section $2$, we give the proof of Proposition \ref{p1}.\\
- In Section $3$, we give the proof of Theorem \ref{theo3}.\\

\section{Characterization of the set of stationary solutions}
In this section, we prove Proposition \ref{p1} which characterizes all $\mathcal{H}_0$ solutions of
\begin{equation}
 \frac{1}{\rho} (\rho (1-y^2) w')'-\frac{2(p+1)}{(p-1)^2}w+|w|^{p-1}w=0,\label{47}
\end{equation}
the stationary version of (\ref{equa}). Note that since $0$ and $\kappa_0 \Omega$ are trivial solutions to equation (\ref{equa}) for any $\Omega  \in \Bbb S^{m-1}$, we see from a Lorentz transformation (see Lemma 2.6 page 54 in \cite{MR2362418})
 that $\mathcal{T}_d e^{i\theta}\kappa_0=\kappa(d,y)$ is also a stationary solution to (\ref{equa}). Let us introduce the set
\begin{equation}\label{48}
 S\equiv \{0, \kappa (d,.)\Omega, |d|<1, \Omega  \in \Bbb S^{m-1}\}.
\end{equation}
Now, we prove Proposition \ref{p1} which states that there are no more solutions of (\ref{47}) in $\mathcal{H}_0$ outside the set $S$.

\noindent 
We first prove $(ii)$, since its proof is short.\\
(ii) Since we clearly have from the definition (\ref{15}) that $E(0,0)=0$, we will compute $E(\Omega\kappa(d,.),0)$. 
From (\ref{15}) and the proof of the real case treated in page 59 in \cite{MR2362418}, we see that
$$ E(\kappa(d,.)\Omega,0)= E(\kappa(d,.),0)=E(\kappa_0,0)>0.$$
Thus, (\ref{14}) follows.\\
(i) Consider $w \in \mathcal{H}_0$ an $ \mathbb{R}^m$ non-zero solution of (\ref{47}). Let us prove that there are some $d\in (-1,1)$ and $\Omega \in \mathbb{S}^{m-1}$ such that $w=\kappa(d,.)\Omega.$ For this purpose, define
\begin{equation}\label{50}
 \xi=\frac{1}{2} \log\left(\frac{1+y}{1-y}\right) (\mbox{that is }  y=\tanh \xi)\mbox{ and } \bar w( \xi)=w(y) (1-y^2)^\frac{1}{p-1}. 
\end{equation}
As in the real case, we see from straightforward calculations that $\bar w\not\equiv 0$ is a $H^1(\Bbb{R})$ solution to
\begin{equation}\label{51}
 \partial_\xi^2\bar w +|\bar w|^{p-1}\bar w-\frac{4}{(p-1)^2}\bar w=0,\,\forall \xi \in \Bbb{R}.
\end{equation}

 Our aim is to prove the existence of $\Omega \in \mathbb{S}^{m-1}$ and $\xi_0 \in \mathbb{R}$ such that $\bar w(\xi)=\Omega \bar {k} (\xi+\xi_0)$ where 
$$\bar {k}(\xi)=\frac{\kappa_0}{\cosh^\frac{2}{p-1}(\xi)}.$$
Since $\bar w \in H^1(\Bbb R)  \subset C^\frac{1}{2}(\Bbb R),$ we see that $\bar w$ is a strong $C^2$ solution of equation (\ref{51}). Since $\bar w\not\equiv  0$, there exists $\xi_0 \in \mathbb{R}$ such that $\bar w(\xi_0)\neq 0$. By invariance of (\ref{51}) by translation, we may suppose that $\xi_0=0$.
Let
 \begin{equation}\label{*}
  G^*=\left\{ \xi \in \mathbb{R} \,|\, \bar w(\xi)\neq 0 \right\},
 \end{equation}
 a nonempty open set by continuity. Note that $G^*$ contains some non empty interval $I$ containing 0.

 We introduce $\rho$ and $\Omega$ by $$\rho=|\bar w|,\,\Omega=\frac{\bar w}{|\bar w|},\mbox{ whenever }\xi\in G^*.$$ 
From equation (\ref{51}), we see that
 \begin{equation}\label{17.5}
\rho''\Omega+2 \rho' \Omega'+ \rho \Omega''+\rho^p\Omega- \frac{4}{(p-1)^2} \rho\Omega=0.
 \end{equation}
Now, since $|\Omega|=1$, we immediately see that $ \Omega'. \Omega=0$ and $ \Omega''. \Omega+|\Omega'|^2=0$.

Let $H(\xi)=|\Omega'|^2$. Projecting equation (\ref{17.5}) according to $\Omega$ and $\Omega'$ we see that
 \begin{equation}\label{mer}
 \forall \xi \in G^*,\,\left\{
\begin{array}{l}
 \rho''(\xi)-\rho(\xi) H(\xi) -c_0\rho(\xi)+\rho(\xi)^p=0,\; c_0= \frac{4}{(p-1)^2}\\
4\rho'(\xi)H(\xi)+\rho(\xi) H'(\xi)=0
\end{array}
 \right . 
\end{equation}
 Integrating the second equation on the interval $I\subset G^*$, we see that for all $\xi \in I$, $H(\xi)=\frac{H(0)(\rho(0))^4}{(\rho(\xi))^4}$. Plugging this in the first equation, we get
 \begin{equation}\label{stka}
 \forall \xi \in I,\,\rho''(\xi)-\frac{\mu}{(\rho(\xi))^3}-c_0\rho(\xi)+\rho^p(\xi)=0 \mbox{ where }\mu=H(0)(\rho(0))^4.
 \end{equation}
\noindent Now let
 \begin{eqnarray}\label{waw}
 G= \left\{ \xi \in G^*, \forall \xi'\in I_\xi,\, H(\xi')=\frac{H(0)\rho(0)^4}{\rho(\xi')^4} \right\},
 \end{eqnarray}
where $I_\xi=[0,\xi)$ if $\xi\ge 0$ or $I_\xi=(\xi,0]$ if $\xi\le 0$. Note that $I \subset G$. Now, we give the following:
\begin{lem} \label{tunisie}
 There exists 
$\epsilon_0 >0$ such that 
$$\forall \xi \in G, \,\forall \xi'\in  I_\xi,\, 0< \epsilon_0 \le |\bar w(\xi')|\le \frac{1}{\epsilon_0}.$$

\end{lem}
\begin{proof} The proof is the same as in the complex-case, see page $5898$ in \cite{<3}. But for the reader's convenience and for the sake of self-containedness, we recall it here.
Take $\xi \in G$. By definition (\ref{waw}) of $G$, we see that equation (\ref{stka}) is satisfied for all $\xi'\in I_\xi$. Multiplying $\rho''(\xi)-\frac{\mu}{(\rho(\xi))^3}-c_0\rho(\xi)+\rho^p(\xi)=0  $ by $\rho'$ and integrating between $0$ and $\xi$, we get:
 $$\forall \xi  \in  I_\xi,\, \mathcal{E}(\xi')=\mathcal{E}(0), \mbox{ where }\mathcal{E}(\xi')=\frac{1}{2}(\rho'(\xi'))^2+\frac{\mu}{2( (\rho(\xi'))^2}-\frac{c_0}{2}\rho^2(\xi')+\frac{\rho^{p+1}(\xi')}{p+1},$$
\noindent or equivalently, 
$$\forall \xi' \in I_\xi,\, F(\rho(\xi'))=\frac{1}{2} \rho'(\xi')^2\ge 0 \mbox{ where }F(r)=\frac{\mu}{2r^2}+\frac{c_0}{2} r^2-\frac{r^{p+1}}{p+1}+\mathcal{E}(0).$$
Since $F(r)\rightarrow -\infty$ as $r \rightarrow 0$ or $r \rightarrow \infty$, there exists $\epsilon_0=\epsilon_0(\mu, E(0)) >0$ such that $\epsilon_0\le \rho(\xi') \le \frac{1}{\epsilon_0}$, which yields to the conclusion of the Claim \ref{tunisie}.
\end{proof}
We claim the following:
\begin{lem}\label{os} It holds that
$G=\mathbb{R}$.
\end{lem}
\begin{proof}
Note first that by construction, $G$ is a nonempty interval (note that $0\in I \subset G$ where $I$ is defined right before (\ref{*})). We have only to prove that $\sup G=+\infty$, since the fact that $\inf G=-\infty$ can be deduced by replacing $\bar w(\xi)$ by $\bar w(-\xi)$.\\By contradiction, suppose that $\sup G=a<+\infty$.

First of all, by Lemma \ref{tunisie}, we have for all $ \xi' \in [0,a),  0< \epsilon_0 \le |\bar w(\xi')|\le \frac{1}{\epsilon_0}$. By continuity, this holds also for $\xi'=a$, hence, $\bar w(a)\neq 0$, and $a\in G^*$. Furthermore, by definition of $G$ and continuity, we see that 
\begin{equation}\label{gg}
\forall \xi \in [0, a], H(\xi)=\frac{H(0)\rho(0)^4}{\rho(\xi)^4}.
\end{equation}

Therefore, we see that $a\in G$.
By continuity, we can write for all $\xi \in (a-\delta,a+\delta),$ where $\delta >0 $ is small enough,
 \begin{equation*} \left\{
\begin{array}{l}
 \rho''(\xi)-\rho (\xi)H(\xi)-c_0\rho(\xi)+\rho(\xi)^p=0,\; c_0= \frac{4}{(p-1)^2}\\
4\rho'(\xi)H(\xi)+\rho(\xi) H'(\xi)=0.
\end{array}
 \right . 
\end{equation*}
From the second equation and (\ref{gg}) applied with $\xi=a$, we see that $H(\xi)=\frac{H(a)(\rho(a))^4}{(\rho(\xi))^4}=\frac{H(0)(\rho(0))^4}{(\rho(\xi))^4}$. Therefore, it follows that $(a, a+\delta )\in G$, which contradicts the fact that $a=\sup G$.
\end{proof}
Note from Lemma \ref{os} that (\ref{mer}) and (\ref{stka}) holds for all $\xi \in \Bbb{R}$.
 We claim that $H(0)=0$. Indeed, if not, then by (\ref{stka}), we have $\mu \neq 0$, and since $G=\Bbb{R}$, we see from Lemma \ref{tunisie} that for all $\xi\in\Bbb{R}$, $|\bar w(\xi)|\ge \epsilon_0$, therefore $w \notin L^2(\Bbb{R})$, which contradicts the fact that $\bar w\in H^1(\mathbb{R})$. Thus, $H(0)=0$, and $\mu=0$.
 By uniqueness of solutions to the second equation of (\ref{mer}), we see that $H(\xi)=0$ for all $\xi \in \Bbb{R}$, so $\Omega(\xi)=\Omega(0)$, and

 \begin{equation*} \left\{
\begin{array}{l}
\bar  w(0)=\rho(0) \Omega(0)\\
\bar  w'(0)=\rho'(0)  \Omega(0).
\end{array}
 \right . 
\end{equation*}
Let $W$ be the maximal real-valued solution of
 \begin{equation*} \left\{
\begin{array}{l}
W''-c_0 W+|W|^{p-1}W=0\\
 W(0)=\rho(0)\\
 W'(0)=\rho'(0).
\end{array}
 \right . 
\end{equation*}
By uniqueness of the Cauchy problem of equation (\ref{51}), we have for all $\xi \in \mathbb{R}, \bar w(\xi)=W(\xi)  \Omega(0)$, and as $\bar w\in H^1(\mathbb{R})$, $W$ is also in $H^1(\mathbb{R})$. It is then classical that there exists $\xi_0$ such that for all $\xi \in \mathbb{R}$, $W(\xi)=\bar {k }(\xi+\xi_0)$ (remember that $\rho(0)>0$, hence we only select positive solutions here). In addition, for $ \Omega_0= \Omega(0)$, $\bar w(\xi)=\bar {k }(\xi+\xi_0)\Omega_0$.
Thus, for $d=\tanh \xi_0\,\in\, (-1,1)$ and $y=\tanh \xi$, we get
\begin{align*}
 &\bar w(\xi)=\kappa_0 \left[1-\tanh (\xi+\xi_0)^2\right]^\frac{1}{p-1}\Omega_0=\kappa_0 \left[1-\left(\frac{\tanh \xi+\tanh \xi_0}{1+\tanh \xi \tanh \xi_0}\right)^2\right]^\frac{1}{p-1}\Omega_0\notag\\&= \kappa_0 \left[1-\left(\frac{y+d}{1+dy}\right)^2\right]^\frac{1}{p-1}\Omega_0= \kappa_0 \left[\frac{(1-d^2)(1-y^2)}{(1+dy)}^2\right]^\frac{1}{p-1}\Omega_0= \kappa(d,y) (1-y^2)^\frac{1}{p-1}\Omega_0.
\end{align*}
By (\ref{50}), we see that $w(y)= \kappa(d,y)\Omega_0$. This concludes the proof of Proposition \ref{p1}.

\medskip
\section{Outline of the proof of Theorem \ref{theo3}}\label{section5}
The proof of Theorem \ref{theo3} is not a simple adaptation of the complex-case to the vector-valued case, in fact, it involves a delicate modulation.
In this section, we will outline the proof, insisting on the novelties, and only recalling the features which are the same as in the real-valued complex-valued cases.

\medskip
This section is organized as follows:\\
- In Subsection 3.1, we linearize equation (\ref{equa}) around $\kappa(d,y)e_1$ where $e_1=(1,0,...,0)$ and figure-out that, with respect to the complex-valued case, our linear operator is just a superposition of one copy of the real part operator, with $(m-1)$ copies of the imaginary part operator. \\
- In Subsection 3.2, we recall from \cite{MR2362418} the spectral properties of the real-part operator.\\
-  In Subsection 3.3,  we recall from \cite{<3} the spectral properties of the imaginary-part operator.\\
- In Subsection 3.4, assuming that $\Omega^*=e_1$ (possible thanks to rotation invariance of (\ref{equa})), we introduce a modulation technique adapted to the vector-valued case. This part makes the originality of our work with respect to the complex-valued case.\\
- In Subsection 3.5, we write down the equations satsified by the modulation parameters along with the PDE satisfied by $q(y,s)$ and its components.\\
- In Subsection 3.6, we conclude the proof of Theorem  \ref{theo3}.

\subsection{ The linearized operator around a non-zero stationary solution}
We study the properties of the linearized operator of equation (\ref{equa}) around the stationary solution $\kappa(d,y)$ $(\ref{defk})$.

\medskip
\noindent Let us introduce $q=(q_1,q_2)\in \mathbb{R}^m\times \mathbb{R}^m$
for all $ s\in[s_0, \infty)$, for a given $s_0 \in \Bbb{R}$, by
\begin{eqnarray}\label{158.5}
\begin{pmatrix} w(y,s)\\\partial_s w(y,s) \end{pmatrix} =\begin{pmatrix} \kappa(d,y)e_1\\0\end{pmatrix} +\begin{pmatrix} q_1(y,s)\\q_2(y,s)\end{pmatrix}.
\end{eqnarray}
Let us introduce the coordinates of $q_1$ and $q_2$ by $q_1=(q_{1,1},q_{1,2},...,q_{1,m})$, $q_2=(q_{2,1},q_{2,2},...,q_{2,m})$. We see from equation (\ref{equa}), that $q$ satisfies the following equation for all $s\ge s_0$:
\begin{eqnarray} \label{linearise-complexe}
\frac{\partial }{\partial s} \begin{pmatrix} q_1\\q_2 \end{pmatrix} = L_d\begin{pmatrix}q_1\\q_2\end{pmatrix}+\begin{pmatrix}0\\ f_d(q_1)\end{pmatrix},
\end{eqnarray}
where
$$ L_d\begin{pmatrix}q_1\\q_2\end{pmatrix}=\begin{pmatrix}q_2 \\\mathcal{L}q_1+\bar\psi (d, y)q_{1,1}e_1+\sum_{j=2}^{m} \tilde\psi (d, y)q_{1,j}e_j-\frac{p+3}{p-1} q_2- 2 y \partial_y q_2 \end{pmatrix}, $$

\begin{eqnarray}\label{barpsi}
\bar \psi(d,y)  =p \kappa(d, y)^{p-1}-\frac{2(p+1)}{(p-1)^2}
\end{eqnarray}
\begin{eqnarray}\label{tildepsi}
\tilde \psi(d,y)  =  \kappa(d, y)^{p-1}-\frac{2(p+1)}{(p-1)^2}
\end{eqnarray}
$$f_d(q_1)={f_{d,1}}(q_1)e_1+ \sum_{j=2}^{m} f_{d,j}(q_1)e_j,$$

where
\begin{eqnarray}\label{barf_d}
 f_{d,1}(q_1)=|\kappa(d, y)e_1+q_1|^{p-1}(\kappa(d, y)+q_{1,1})-\kappa(d, y)^{p}-p\kappa^{p-1}(d, y)q_{1,1}.
\end{eqnarray}
\begin{eqnarray}\label{tildef_d}
 f_{d,j}( q_1)=|\kappa(d, y)e_1+q_1|^{p-1}q_{1,j}-\kappa^{p-1}(d,y)q_{1,j}.
\end{eqnarray}

Projecting (\ref{linearise-complexe}) on the first coordinate, we get for all $s\ge s_0$:

\begin{eqnarray}\label{159.5}
\label{mult}\frac{\partial }{\partial s}\begin{pmatrix} q_{1,1}\\q_{2,1}
\end{pmatrix}= \bar L_d \begin{pmatrix} q_{1,1} \\q_{2,1} \end{pmatrix}+\begin{pmatrix}0\\ f_{d,1}(q_1)\end{pmatrix},
\end{eqnarray}
where $\bar L_d $ is given by:
\begin{eqnarray}\label{barL_d}
\bar L_d  \begin{pmatrix} q_{1,1}\\q_{2,1} \end{pmatrix}=
\begin{pmatrix}q_{2,1} \\\mathcal{L} q_{1,1}+ \bar \psi(d,y)   q_{1,1}-\frac{p+3}{p-1} q_{2,1} - 2 y \partial_y q_{2,1}
\end{pmatrix},
\end{eqnarray}

Now, projecting equation (\ref{linearise-complexe}) on the j-th coordinate with $j=2,..,m,$ we see that
\begin{eqnarray}\label{multi}\frac{\partial }{\partial s}\begin{pmatrix} q_{1,j}\\q_{2,j}
\end{pmatrix}= \tilde L_d\begin{pmatrix} q_{1,j} \\q_{2,j} \end{pmatrix}+\begin{pmatrix}0\\ f_{d,j}( q_1)\end{pmatrix},
\end{eqnarray}
where
\begin{eqnarray}\label{tildeL_d}
 \tilde L_d \begin{pmatrix} q_{1,j}\\q_{2,j} \end{pmatrix}=
\begin{pmatrix}q_{2,j} \\\mathcal{L} q_{1,j}+ \tilde\psi(d,y) q_{1,j}-\frac{p+3}{p-1} q_{2,j} - 2 y \partial_y q_{2,j}
\end{pmatrix},
\end{eqnarray}
{\bf Remark:} Our linearized operator $L_d$ is in fact diagonal in the sens that
$$ L_d\begin{pmatrix}q_1\\q_2\end{pmatrix}=\bar L_d  \begin{pmatrix} q_{1,1}\\q_{2,1} \end{pmatrix}e_1+   \sum_{j=2}^{m}   \tilde L_d \begin{pmatrix} q_{1,j}\\q_{2,j} \end{pmatrix}e_j.$$

We mention that for $j=1$, equation (\ref{159.5}) is the same as the equation satisfied by the real part of the solution in the complex case (see Section $3$ page $5899$ in \cite{<3}), whereas for $j=2,..,m$, equation (\ref{multi}) is the same as the equation satisfied by the imaginary part of the solution operator in the complex case. Thus, the reader will have no difficulty in adapting the remaining part of the proof to the vector-valued case. Thus, the dynamical system formulation we performed when $m=2$ can be adapted straightforwardly to the case $m\ge3$.

 \noindent Note from (\ref{9}) that we have
$$||q||_\mathcal{H}=[\phi (q, q)]^\frac{1}{2}<+\infty,$$
where the inner product $\phi$ is defined by
\begin{equation*}
 \phi (q, r)=\phi \left(\begin{pmatrix} q_1\\q_2 \end{pmatrix}, \begin{pmatrix} r_1\\r_2 \end{pmatrix}\right)=\int_{-1}^{1}(q_1.r_1+q'_1. r'_1(1-y^2)+q_2.r_2)\rho\;dy.
\end{equation*}
where $q_1.r_1=\sum_{j=1}^m q_{1,j}.r_{1,j}$ is the standard inner product in $\mathbb{R}^m$, with similar expressions for $q'_1.r'_1$ and $q_2.r_2$.

Using integration by parts and the definition of $\mathcal{L}$ (\ref{8}), we have the following:
\begin{equation}\label{phi}
 \phi (q, r)=\int_{-1}^{1}(q_1\cdot (-\mathcal{L} r_1+  r_1)+q_2 \cdot r_2)\rho \,dy.
\end{equation}
In the following two sections, we recall from \cite{MR2362418} and \cite{<3} the spectral properties of $\bar L_d$ and $\tilde L_d$.
\subsection{Spectral theory of the operator $\bar L_d$}
From Section 4 in \cite{MR2362418}, we know that $\bar L_d$ has two nonnegative eigenvalues $\lambda=1$ and $\lambda=0$ with eigenfunctions
\begin{equation}
\bar F_1^d (y)= (1-d^2)^{\frac{p}{p-1}}\begin{pmatrix}(1+dy)^{-\frac{p+1}{p-1}}\\(1+dy)^{-\frac{p+1}{p-1}}\end{pmatrix}
\mbox{and }\;
\bar F_0^d (y)= (1-d^2)^{\frac{1}{p-1}}\begin{pmatrix}\frac{y+d}{(1+dy)^\frac{p+1}{p-1}}\\0\end{pmatrix}. \label{110}
\end{equation}
Note that for some $C_0>0$ and any $\lambda\in \{0,1\}$, we have
\begin{equation}\label{majoration}
\forall |d|<1,\;\; \frac{1}{C_0}\leq ||\bar F_\lambda^d||_{\mathcal{H}} \leq C_0 \;\mbox{  and  }\; ||\partial_d \bar F_\lambda^d||_{\mathcal{H}} \leq \frac{C_0}{1-d^2}.
\end{equation}
Also, we know that $\bar L_d^*$ the conjugate operator of $\bar L_d$ with respect to $\phi$ is given by
\begin{equation*}
\bar L_d^*
\begin{pmatrix}
r_1\\
r_2\end{pmatrix}= \begin{pmatrix}
\bar R_d(r_2)\\
-\mathcal{L} r_1+r_1+\frac{p+3}{p-1}r_2+2y r_2'-\frac{8}{(p-1)}\frac{r_2}{(1-y^2)}\end{pmatrix}
\end{equation*}
for any $(r_1, r_2)\in (\mathcal{D}(\mathcal{L}))^2$, where $r=\bar R_d (r_2)$ is the unique solution of
\begin{equation*}
-\mathcal{L} r+r=\mathcal{L} r_2+\bar{\psi} (d, y)r_2. 
\end{equation*}
Here, the domain $\mathcal{D}(\mathcal{L})$ of $\mathcal{L}$ defined in (\ref{8}) is the set of all $r \in L_{\rho}^2$ such that $\mathcal{L} r \in L_{\rho}^2.$
\medskip

\noindent Furthermore, $\bar L_d^*$ has two nonnegative eigenvalues $\lambda=0$ and $\lambda=1$ with eigenfunctions $\bar W_{\lambda}^d$ such that 
\begin{equation*}
\bar W_{1, 2}^d (y)= \bar c_1 \frac{(1-y^2)(1-d)^\frac{1}{p-1}}{(1+dy)^\frac{p+1}{p-1}},\,\bar W_{0, 2}^d (y)= \bar c_0 \frac{(y+d)(1-d)^\frac{1}{p-1}}{(1+dy)^\frac{p+1}{p-1}},
\end{equation*}
with\footnote{ In section 4 of \cite{MR2362418}, we had non explicit normalizing constants $\bar c_\lambda=\bar c_\lambda (d)$. In Lemma 2.4 in \cite{MZ13}, the authors compute the explicit dependence of $\bar c_\lambda (d)$.}
\begin{equation*}
\frac{1}{\bar c_\lambda}=2(\frac{2}{p-1}+\lambda)\int_{-1}^1 (\frac{y^2}{1-y^2} )^{1-\lambda}\rho(y) \,dy,
\end{equation*}
and $\bar W_{\lambda, 1}^d $ is the unique solution of the equation 
\begin{equation*}
-\mathcal{L} r+ r =\left(\lambda-\frac{p+3}{p-1}\right) r_2- 2 y r'_2 + \frac{8}{p-1} \frac{r_2}{1-y^2}
\end{equation*}
with $r_2= \bar W_{\lambda, 2}^d$. We also have for $\lambda=0,1$
\begin{equation}\label{36,5}
||\bar W_\lambda ^d ||_{\mathcal{H}}+ (1-d^2)||\partial_d \bar W_\lambda ^d ||_{\mathcal{H}} \le C, \forall  |d|<1.
\end{equation}

 Note that we have the following relations for $\lambda=0$ or $\lambda=1$
\begin{equation}\label{36,6}
\phi (\bar W_\lambda ^d,\bar{F_\lambda ^d})=1\mbox{ and }\phi (\bar W_\lambda ^d,\bar{F_{1-\lambda} ^d})=0.
\end{equation}
Let us introduce for $\lambda \in \{0, 1\}$ the projectors $\bar \pi_\lambda(r)$, and $\bar \pi_-^d(r)$ for any $r\in\mathcal{H}$ by
\begin{equation} \label{barpi}
\bar \pi_\lambda^d (r)=\phi (\bar W_\lambda^d, r), 
\end{equation}

\begin{equation} \label{ma}
r=\bar \pi_0^d (r) \bar F_0^d (y)+\bar \pi_1^d (r) \bar F_1^d (y)+ \bar \pi_{-}^d (r),
\end{equation}
and the space
\begin{equation*}
\bar{\mathcal{H}}_{-}^d \equiv \{r \in \mathcal{H} \,|\, \bar \pi_1^d (r)=\bar \pi_0^d(r)=0\}.
\end{equation*}
Introducing the bilinear form
\begin{eqnarray}\label{134'}
 \bar \varphi_{d} (q,r)&=&\int_{-1}^{1} (-\bar \psi(d,y) q_1r_1+q_1' r_1'(1-y^2)+q_2 r_2 ) \rho dy,
\end{eqnarray}
where $\bar \psi(d,y)$ is defined in (\ref{barpsi}), we recall from Proposition 4.7 page 90 in  \cite{MR2362418} that there exists $C_0>0$ such that for all $|d|<1$, for all $r\in \bar{\mathcal{H}}^d_{-},$
\begin{align}\frac{1}{C_0} ||r||_{\mathcal{H} }^2\le \bar \varphi_d (r,r)\le C_0 ||r||_\mathcal{H} ^2.
\label{36,55}
\end{align}
Furthermore, if $r\in \mathcal{H},$ then
\begin{align}\label{36,7}
\frac{1}{C_0} ||r||_{\mathcal{H} }\le\left(|\bar \pi_{0}^d  (r)| +|\bar \pi_{1}^d  (r)| +\sqrt{ \bar \varphi_d  (r_{-},r_{-})}\right)\le C_0 ||r||_\mathcal{H}\mbox{ where }r_{-}=\bar \pi_{-}^d  (r) .
\end{align}
\noindent In the following section we recall from \cite{<3} the spectral properties of $\tilde L_d$.

\subsection{Spectral theory of the operator $\tilde L_d$}
From Section 3 in \cite{<3}, we know that $\tilde L_d$ has one nonnegative eigenvalue $\lambda=0$ with eigenfunction 
\begin{equation}\label{F}
\tilde F_0^d(y)= \begin{pmatrix}\kappa (d,y)\\0\end{pmatrix}.
\end{equation}
Note that for some $C_0>0$ we have
\begin{equation}\label{majoration1}
\forall |d|<1,\;\; \frac{1}{C_0}\leq ||\tilde F_0^d||_{\mathcal{H}} \leq C_0 \;\mbox{  and  }\; ||\partial_d \tilde F_0^d||_{\mathcal{H}} \leq \frac{C_0}{1-d^2}.
\end{equation}

We know also that the operator $\tilde L_d^*$ conjugate of $\tilde L_d$ with respect to $\phi$ is given by
\begin{equation}\label{L*_d}
\tilde L_d^*
\begin{pmatrix}
r_1\\
r_2\end{pmatrix}= \begin{pmatrix}
\tilde R_d(r_2)\\
-\mathcal{L} r_1+r_1+\frac{p+3}{p-1}r_2+2y r_2'-\frac{8}{(p-1)}\frac{r_2}{(1-y^2)}\end{pmatrix}
\end{equation}
for any $(r_1, r_2)\in (\mathcal{D}(\mathcal{L}))^2$, where $r=\tilde R_d (r_2)$ is the unique solution of
\begin{equation}
-\mathcal{L} r+r=\mathcal{L} r_2+\tilde{\psi} (d, y)r_2. \label{sol}
\end{equation}

Furthermore, $\tilde L_d^*$  have one nonnegative eigenvalue $\lambda=0$ with eigenfunction $\tilde W_0^d$ such that
\begin{equation}\label{W}
\tilde W_{0, 2}^d (y)= \tilde c_0 \kappa(d, y) \mbox{ and }\frac{1}{\tilde c_0}=\frac{4\kappa_0^2}{p-1} \int_{-1}^1 \frac{\rho(y)}{1-y^2}dy
\end{equation}
and $\tilde W_{0, 1}^d $ is the unique solution of the equation
\begin{equation}\label{tildeW_1}
-\mathcal{L} r+ r =-\frac{p+3}{p-1} r_2- 2 y r'_2 + \frac{8}{p-1} \frac{r_2}{1-y^2}
\end{equation}
with $r_2= \tilde W_{0, 2}^d$.

We also have for $\lambda=0,1$
\begin{equation}\label{47,5}
||\tilde W_0 ^d ||_{\mathcal{H}}+ (1-d^2)||\partial_d \tilde W_\lambda ^d ||_{\mathcal{H}} \le C,\; \forall  |d|<1.
\end{equation}

Moreover, we have
\begin{equation}\label{1}
\phi (\tilde W_0^d,\tilde{F_0^d})=1.
\end{equation}
Let us introduce the projectors $\tilde \pi_0^d (r)$ and $ \tilde \pi_{-}^d (r))$ for any $r \in \mathcal{H}$ by
\begin{equation} \label{pi}
\tilde \pi_0^d (r)=\phi (\tilde W_0^d, r), 
\end{equation}
\begin{equation} \label{s}
r=\tilde \pi_0^d (r) \tilde F_0^d (y)+ \tilde \pi_{-}^d (r).
\end{equation}

and the space

\begin{equation}\label{ss}
\tilde{\mathcal{H}}_{-}^d \equiv \{r \in \mathcal{H} \,|\,  \tilde \pi_{0}^d (r)=0\}.
\end{equation}

Introducing the bilinear form
\begin{eqnarray}\label{134}
 \tilde \varphi_{d} (q,r)&=&\int_{-1}^{1} (-\tilde \psi(d,y) q_1r_1+q_1' r_1'(1-y^2)+q_2 r_2 ) \rho dy,\notag\\
\end{eqnarray}
where $\tilde \psi(d,y) $ is defined in (\ref{tildepsi}), we recall from Proposition $3.7$ page $5906$ in \cite{<3} that there exists $C_0 >0$ such that for all $|d|<1,$ for all $r\in \tilde{\mathcal{H}}^d_{-},$

\begin{eqnarray}\label{50.5}
 \frac{1}{C_0} ||r||_{\mathcal{H} }^2\le \tilde \varphi_d (r,r)\le C_0 ||r||_\mathcal{H} ^2.
\end{eqnarray}

\subsection{ A modulation technique}\label{sec,mod}

We start the proof of Theorem \ref{theo3} here.\\
Let us consider $w\in C ([s^*,\infty), \mathcal{H})$ for some $s^*\in \mathbb{R}$ a solution of equation (\ref{equa}) such that
$$\forall s\ge s^*, E(w(s),\partial_s w(s))\ge E(\kappa_0,0)$$
and
\begin{eqnarray}\label{56,5}
\Big|\Big|\begin{pmatrix} w(s^*)\\\partial_s w(s^*) \end{pmatrix} -\begin{pmatrix} \kappa(d^*,.)\Omega^*\\0\end{pmatrix} \Big|\Big|_{\mathcal{ H}}\le \epsilon^*
 \end{eqnarray}
for some $d^*\in (-1,1)$, $\Omega^*\in  \mathbb{S}^{m-1}$ and $\epsilon^*>0$ to be chosen small enough.
\medskip

Our aim is to show the convergence of $(w(s),\partial_s w(s))$ as $s\rightarrow \infty$ to some $(\kappa(d_\infty,0)\Omega_\infty,0)$, for some $(d_\infty,\Omega_\infty)$ close to $(d^*,\Omega^*)$.
\medskip

As one can see from (\ref{56,5}), $(w,\partial_s w)$ is close to a one representative of the family of the non-zero stationary solution 
\begin{equation*}\label{}
 S^*\equiv \{ (\kappa (d,y),0)\Omega, |d|<1, \Omega  \in \Bbb S^{m-1}\}.
\end{equation*}
From the continuity of $(w,\partial_s w)$ from $[s^*,\infty)$ to $\mathcal{ H}$, $(w(s),\partial_s w(s))$ will stay close to a soliton from $ S^*$, at least for a short time after $s^*$. In fact, we can do better, and impose some orthogonality conditions, killing the zero directions of the linearized operator of equation (\ref{equa}) (see the operator $L_d$ defined in (\ref{linearise-complexe})).
\medskip

From the invariance of equation (\ref{equa}) under rotations in $\mathbb{R}^{m}$, we may assume that 

\begin{eqnarray}\label{53,5}
\Omega^*=e_1.
\end{eqnarray}
 \medskip
 
We recall that at this level of the study in the complex case (i.e. for $m=2$), we were able to modulate $(w,\partial_s w)$ as follows
\begin{eqnarray}\label{168}
\begin{pmatrix} w(y,s)\\\partial_s w(y,s) \end{pmatrix} =e^{i \theta (s)}\left[\begin{pmatrix} \kappa(d(s),y)\\0\end{pmatrix} +\begin{pmatrix} q_1(y,s)\\q_2(y,s)\end{pmatrix}\right].
\end{eqnarray}
for some well chosen $d(s)\in (-1,1)$ and $\theta(s)\in \mathbb{R}$, such that 
\begin{equation*}
\bar\pi_0^{d(s)}\begin{pmatrix} q_{1,1}(s)\\q_{2,1}(s)\end{pmatrix} =\tilde\pi_0^{d(s)}\begin{pmatrix} q_{1,2}(s)\\q_{2,2}(s)\end{pmatrix}=0
\end{equation*}
where $\bar\pi_0^{d}$ and $\tilde\pi_0^{d}$ are defined in (\ref{barpi}) and (\ref{pi}) and $q=(q_1, q_2)$ is small in $\mathcal{ H}$.

From (\ref{168}), we see that we have a rotation in the complex plane, which has to be generalized to the vector-valued case. In order to do so, we introduce for $i=2,...,m$

\begin{eqnarray}\label{matrice}
 R_i\equiv \begin{pmatrix}
\cos \theta_i&0&\cdots&-\sin\theta_i &\cdots&0\\
0 &1&\cdots&0 &\cdots&0\\
\vdots&\vdots&\ddots&\vdots&\ddots&\vdots\\
\sin \theta_i&0&\cdots&\cos\theta_i &\cdots&0\\
\vdots&\vdots&\ddots&\vdots&\ddots&\vdots\\
0 &0&\cdots&0&\cdots&1\\
\end{pmatrix}.
\end{eqnarray}
Note that $R_i$ is an $m\times m$ orthonormal matrix which rotates the $(e_1,e_i)$-plane by an angle $\theta_i$ and leaves all other directions invariant. We introduce $ R_\theta$ by
\begin{eqnarray}\label{R_theta}
 R_\theta\equiv R_2 R_3\cdots R_m,
\end{eqnarray}
where $\theta=(\theta_2,\theta_3,\cdots , \theta_m)$. Clearly, $R_\theta$ is an $m\times m$ orthonormal matrix.
We also define $A_j$ by
\begin{eqnarray}\label{A_i}
A_j=R_{\theta}^{-1} \frac{\partial R_\theta}{\partial \theta_j}.
\end{eqnarray}
In the appendix, we show a different expression for $A_j$:
\begin{eqnarray}\label{61,5}
A_j=\frac{\partial R_\theta^{-1} }{\partial \theta_j} R_{\theta}.
\end{eqnarray}

In fact, this formalism is borrowed from Filippas and Merle \cite{MR1317705} who introduced the modulation technique for the vector-valued heat equation
$$\partial_t u=\Delta u+|u|^{p-1}u.$$
We are ready to give our modulation technique result well adapted to the vector-valued case:

\begin{prop}{\bf (Modulation of $w$ with respect to $\kappa (d, .)\Omega$, where $\Omega\in \mathbb{R}^{m-1}$)}\label{5.11}
There exists $\epsilon_0 > 0$ and $K_1 >0$ such that for all $\epsilon\le \epsilon_0$ if $v\in\mathcal{H}$, $ d\in (-1,1) $ and $\hat{\theta}=(\hat{\theta}_2,...,\hat{\theta}_{m})\in  \mathbb{R}^{m-1}$  are such that
\begin{eqnarray*}\label{}
\forall i=2,...,m, \cos {\hat{\theta}}_i\ge \frac{3}{4}\mbox{ and } ||\hat{q}||_{\mathcal{H}}\le  \epsilon \mbox{ where } 
v =R_{\hat{\theta }}\left[\begin{pmatrix} \kappa(\hat d,.)e_1\\0\end{pmatrix} +\hat{q}\right],
\end{eqnarray*}
then, there exist $d\in (-1,1),$  $\hat{\theta}=(\hat{\theta}_2,...,\hat{\theta}_{m})\in  \mathbb{R}^{m-1}$ such that
\begin{equation}\label{167}
\bar\pi_{0}^{d}\begin{pmatrix} q_{1,1}\\ q_{2,1}\end{pmatrix} =0,\mbox{ and }\tilde\pi_{0}^{d}\begin{pmatrix} q_{1,j}\\ q_{2,j}\end{pmatrix} =0,\; \forall j=2,..m,
\end{equation}
where $q=(q_1, q_2)$ is defined by:

\begin{equation*}
\Big| \log\left(\frac{1+d}{1-d}\right)-\log\left(\frac{1+\hat{d}}{1-\hat{d}}\right)\Big|+ |\theta-\hat\theta|\le C_0 ||\hat q||_{\mathcal{H}}\le K_1 \epsilon,
\end{equation*}
\begin{equation*}
 \forall i=2,...,m,\; \cos \theta_i\ge \frac{1}{2}\mbox{ and }||q||_{\mathcal{H}}\le K_1 \epsilon.
\end{equation*}
\end{prop}

In order to prove this proposition, we need the following estimates on the matrix $A_j$ given in $(\ref{A_i})$ and $(\ref{61,5})$:
\begin{lem}[{\bf Orthogonality and continuity results related to the matrix $A_i$ (\ref{A_i})}]\label{prod}$ $\\
i) For any $i\in \{2,...,m\}$, 
 \begin{eqnarray*}
 A_i e_1=( \prod\limits_{j =i+1}^{m} \cos\theta_j ) e_i
 \end{eqnarray*}
ii) For any  $i\in \{2,...,m\}$, $z\in \mathbb{R}^m$, we have
$$ |A_i (z) |\le |z| .$$
\end{lem}
\begin{proof}
The proof is straightforward though a bit technical. For that reason, we give it in Appendix \ref{AA}
\end{proof}

Now, we are ready to prove Proposition  \ref{5.11}.
\begin{proof}[Proof of Proposition \ref{5.11}]
 The proof is similar to the complex-valued case. However, since our notations are somehow complicated, we give details for the reader's convenience.\\
  First, we recall that $\theta=(\theta_2,\theta_3,...,\theta_m)\in\mathbb{R}^{m-1}$. \\From (\ref{barpi}) and (\ref{pi}), we see that the condition (\ref{167}) becomes $\Phi(v, d, \theta )=0 $ 
where $ \Phi  \in C(\mathcal{H}\times (-1,1)\times\mathbb{R}^{m-1}, \mathbb{R}^m) $ is defined by
\begin{equation}
\begin{array}{c}\label{172}
\Phi (v, d, \theta)=\begin{pmatrix}\bar \Phi (v, d, \theta)\\\tilde \Phi_2 (v, d, \theta)\\\vdots\\\tilde \Phi_m (v, d, \theta)\end{pmatrix}=\begin{pmatrix}\phi\left(\begin{pmatrix}V_{1,1}\\V_{2,1} \end{pmatrix}-\begin{pmatrix}\kappa(d,.)\\0 \end{pmatrix},\bar W_0^d\right)\\
  \phi\left(\begin{pmatrix}V_{1,j}\\V_{2,j} \end{pmatrix},\tilde W_0^d\right)_{j=2...m}                                                                                                                                                                                                                                                                                                                            

\end{pmatrix}
\end{array}
\end{equation}
where $V=\begin{pmatrix}
          V_1\\V_2
         \end{pmatrix}
\in\mathbb{R}^m\times\mathbb{R}^m$ is given by $V=R^{-1}_\theta v$.\\
We claim that we can apply the implicit function theorem to $\Phi$ near the point $(\hat{v},\hat{d},\hat{\theta})$ with
$\hat{v}=R_{\hat \theta}(\kappa(\hat d,.)e_1,0)$.  Three facts have to be checked:\\
1-First, note that $\hat v=R_{\hat \theta}^{-1}(\hat v)$, hence
\begin{equation*}
\Phi (R_{\hat \theta}(\kappa(\hat d,.)e_1,0),\hat d,\hat \theta
)=0.
\end{equation*}
2-Then, we compute from (\ref{172}), for all $u\in \mathcal{H}$, 
\begin{equation*}
 D_v \bar \Phi (v,d,\theta)(u)=\phi(\begin{pmatrix}U_{1,1}\\U_{2,1} \end{pmatrix},\bar W_0^d),
\end{equation*}
and for all $ j=2...m$, we have
\begin{equation*}
D_v \tilde \Phi_j (v,d,\theta)(u)=\phi(\begin{pmatrix}U_{1,j}\\U_{2,j} \end{pmatrix},\tilde W_0^d) ,
\end{equation*}
so we have from (\ref{36,5}) and (\ref{47,5})
\begin{equation}\label{175}
 ||D_v \bar \Phi (v,d,\theta)||\le C_0 \mbox{ and } ||D_v \tilde \Phi_j (v,d,\theta)||\le C_0.
\end{equation}\\
3-Let $J(\bar \Phi,\tilde \Phi_{j,j=2..m})$ the jacobian matrix of $\Phi$ with respect to $(d,\theta)$, and $D$ its determinant so
\begin{eqnarray}\label{jac}
 J\equiv \begin{pmatrix}
\partial_d \bar \phi&\partial_{\theta_2} \bar \phi&\cdots &\partial_{\theta_m} \bar \phi\\
\partial_d \tilde \phi_2&\partial_{\theta_2} \tilde \phi_2&\cdots &\partial_{\theta_m} \tilde \phi_2\\
 \vdots&\vdots&\vdots&\vdots\\
\partial_d \tilde \phi_m&\partial_{\theta_2} \tilde \phi_m&\cdots &\partial_{\theta_m} \tilde \phi_m\\
\end{pmatrix}.
\end{eqnarray}
Then, we compute from (\ref{172}):
\begin{eqnarray}
 \label{69,5} \partial_d \bar \Phi&=&-\phi((\partial_d \kappa(d,.),0),\bar W_0^d)+\phi(\begin{pmatrix}V_{1,1}\\V_{2,1} \end{pmatrix}-\begin{pmatrix}\kappa(d,.)\\0 \end{pmatrix}, \partial_d \bar W_0^d),  \end{eqnarray}
 and for $i,j=2,..m$
  \begin{eqnarray}
  \partial_d \tilde \Phi_j&=&\phi(\begin{pmatrix}V_{1,j}\\V_{2,j} \end{pmatrix},\partial_d \tilde W_0^d),\label{ddj}\\
 \partial_{\theta_i} \bar \Phi&=&\phi( \partial_{\theta_i} \begin{pmatrix}V_{1,1}\\V_{2,1} \end{pmatrix},\bar W_0^d)=\phi(  \begin{pmatrix}  <e_1,  \frac{\partial R_\theta^{-1}}{ \partial{\theta_i} }v_{1}>\\ <e_1,  \frac{\partial R_\theta^{-1}}{ \partial{\theta_i} }v_{2}> \end{pmatrix},\bar W_0^d),\label{dtj}\\
\partial_{\theta_i} \tilde \Phi_j&=&\phi(\begin{pmatrix}  <e_j,  \frac{\partial R_\theta^{-1}}{ \partial{\theta_i} }v_{1}>\\ <e_j,  \frac{\partial R_\theta^{-1}}{ \partial{\theta_i} }v_{2}> \end{pmatrix},\partial_d \tilde W_0^d).\label{dtjj}\\
\end{eqnarray}

Now, we assume that
  \begin{eqnarray}\label{avion2}
|\theta|+\big|\log \left(\frac{1+d}{1-d}\right)-  \log \left(\frac{1+\hat d}{1-\hat d}\right)\big|+ \big|\big|  v -R_{\hat \theta}\begin{pmatrix}    \kappa(\hat d,.)e_1\\0     \end{pmatrix}  \big|\big|_{\mathcal{H} }\le \epsilon_1
  \end{eqnarray}
for some small $\epsilon_1>0$.\\
In the following, we estimate each of the derivatives whose  expressions where given above.\\
- Since $$\begin{pmatrix}   \partial_d \kappa(d,y)\\0\end{pmatrix}=\frac{-2\kappa_0}{(p-1)(1-d^2)}\bar F_0^d,$$
by definiftion (\ref{110}) and (\ref{defk}), it follows from the orthogonality condition (\ref{36,6}) that

$$\phi((\partial_d \kappa(d,.),0),\bar W_0^d)=\frac{-2\kappa_0}{(p-1)(1-d^2)}.$$
Therefore, from (\ref{69,5}), we write
\begin{eqnarray}
 \label{69,55} \partial_d \bar \Phi=\frac{2\kappa_0}{(p-1)(1-d^2)}+\phi(\begin{pmatrix}V_{1,1}\\V_{2,1} \end{pmatrix}-\begin{pmatrix}\kappa(d,.)\\0 \end{pmatrix}), \partial_d \bar W_0^d).
  \end{eqnarray}
  Since $\begin{pmatrix}V_{1,1}\\V_{2,1} \end{pmatrix}=\begin{pmatrix}  <e_1,  R_\theta^{-1}v_{1}>\\ <e_1,   R_\theta^{-1}v_{2}> \end{pmatrix},$ we write
\begin{eqnarray}
 \label{69,6} &\;&\big|\big| \begin{pmatrix}V_{1,1}\\V_{2,1} \end{pmatrix}- \begin{pmatrix}\kappa(d,.)\\0 \end{pmatrix}  \big|\big|_{\mathcal{H} }\le    \big|\big| R_\theta^{-1}v-\begin{pmatrix}\kappa(d,.)e_1\\0 \end{pmatrix}  \big|\big|_{\mathcal{H} }\\\notag &\le&\big|\big|           ( R_\theta^{-1}- R_{\hat \theta}^{-1})  v \big|\big|_{\mathcal{H} }+\big|\big|   R_{\hat \theta}^{-1}  v-\begin{pmatrix}\kappa(\hat d,.)e_1\\0 \end{pmatrix}     \big|\big|_{\mathcal{H} }+  \big|\big|  \kappa(\hat d,.)-\kappa(d,.)      \big|\big|_{\mathcal{H} }.
  \end{eqnarray}
Since,
$$\forall \theta, \theta' \in \mathbb{R}, \big| R_{\theta}-R_{\theta'}\big|+\big| R_{\theta}^{-1}-R_{\theta'}^{-1}\big| \le C  \big| \theta-\theta' \big|, $$
(see (\ref{jac}) below for $R_{\theta}$, and use an adhoc change of variables for $R_{\theta}^{-1}$), recalling the following continuity result from estimate $(174)$ page $101$ in \cite{MR2362418}:
\begin{eqnarray}\label{avion}
\big|\big|   \kappa(d_1,.)-  \kappa(d_2,.)    \big|\big|_{\mathcal{H}_0 }\le C\big|  \left(\frac{1+d_1}{1-d_1}\right)- \left(\frac{1+d_2}{1-d_2}\right)\big| ,
  \end{eqnarray}
we see from the Cauchy-Schwartz inequality, (\ref{69,6}), (\ref{36,5}) and (\ref{avion2}) that
\begin{eqnarray}
 \label{999} \big|\big| \begin{pmatrix}V_{1,1}\\V_{2,1} \end{pmatrix}- \begin{pmatrix}\kappa(d,.)\\0 \end{pmatrix}  \big|\big|_{\mathcal{H} }\le    \big|\big| V-\begin{pmatrix}\kappa(d,.)e_1\\0 \end{pmatrix}  \big|\big|_{\mathcal{H} }\            \le C \epsilon_1.
  \end{eqnarray}
\begin{eqnarray}
 \label{1000} \big|  \partial_d \bar \Phi-\frac{2\kappa_0}{(p-1)(1-d^2)}\big|  \le C \epsilon_1.
  \end{eqnarray}
  - Since
$$ \begin{pmatrix}V_{1,j}\\V_{2,j} \end{pmatrix}=\begin{pmatrix}  <e_j,  R_\theta^{-1}v_{1}>\\ <e_j,   R_\theta^{-1}v_{2}> \end{pmatrix},$$
we write
\begin{eqnarray}
   \big|\big|    \begin{pmatrix}V_{1,j}\\V_{2,j} \end{pmatrix}     \big|\big|_{\mathcal{H} }\le \big|\big|   R_{\theta}^{-1}  v-\begin{pmatrix}\kappa(d,.)e_1\\0 \end{pmatrix}     \big|\big|_{\mathcal{H} }\le C \epsilon_1
  \end{eqnarray}
by the same argument as for (\ref{69,6}). Using the Cauchy-Schwarz inequality together with (\ref{47,5}), we see from (\ref{ddj}) that 
\begin{eqnarray}
 \label{1002} \big|  \partial_d \tilde \Phi_j\big|  \le \frac{C \epsilon_1}{1-d^2}.
  \end{eqnarray}
From (\ref{61,5}), we see that $\frac{\partial R_\theta^{-1}}{\partial \theta_i}= A_i R_\theta^{-1}$. Therefore using $ii)$ of Lemma \ref{prod} and the fact that the rotation $R_\theta$ does not change the norm in $\mathcal{H}$, we write
\begin{eqnarray}
 \label{} \big|     \big| \frac{\partial R_{\theta}^{-1}}{\partial \theta_i}(v)-  \begin{pmatrix}\kappa(d,.)A_i (e_1)\\0 \end{pmatrix}              \big|\big|_{\mathcal{H} }\le \big| \big| v- R_{\theta} \begin{pmatrix}\kappa(d,.)e_1\\0 \end{pmatrix} \big|  \big| _{\mathcal{H} }\le C \epsilon_1,
  \end{eqnarray}
  by the same argument as for (\ref{69,6}). Using the Cauchy-Schwarz identity together with (\ref{36,5}), we see from (\ref{dtj}) that
  \begin{eqnarray}
 \label{1001} \big|  \partial_{\theta_i} \bar \Phi\big|  \le C \epsilon_1.
  \end{eqnarray}
  - By the same argument as for (\ref{1001}), we obtain from (\ref{dtjj})
   \begin{eqnarray}
 \label{1003} \big|  \partial_{\theta_i} \tilde \Phi_j\big|  \le C \epsilon_1 \mbox{ if }i\neq j.
  \end{eqnarray} 
  Now, if $i= j$, noting from (\ref{61,5}) that $\frac{\partial R_\theta^{-1}}{\partial \theta_i}v=A_i R_\theta^{-1}(v)=A_i  V,$ applying the operator $A_i$ to (\ref{999}), then taking the scalar product with $e_i$, we see from Lemma \ref{prod} that
  
  \begin{eqnarray*}
 \label{} \big|     \big|  \begin{pmatrix} <e_i, \frac{\partial R^{-1}_{\theta}}{\partial \theta_i} v_{1}>  \\ <e_i, \frac{\partial R^{-1}_{\theta}}{\partial \theta_i} v_{2}>\end{pmatrix}     - \begin{pmatrix}\kappa(d_i,.)\prod\limits_{k=i+1}^{m}\cos \theta_k\\0 \end{pmatrix}            \big|\big|_{\mathcal{H} }\le C \epsilon_1.
  \end{eqnarray*}
  Since we know from (\ref{F}) and (\ref{1}) that
 $$\phi(\kappa(d,.), \tilde W_0^d)=1,$$
it follows from (\ref{dtjj}) that
  \begin{eqnarray}
 \label{1004} \big|  \partial_{\theta_i} \tilde \Phi_i -\prod\limits_{k=i+1}^{m}\cos \theta_k\big|  \le C \epsilon_1.
  \end{eqnarray}
Collecting (\ref{1000}),  (\ref{1002}), (\ref{1001}), (\ref{1003}) and  (\ref{1004}) we see that
  \begin{eqnarray*}
 \label{} \big| D-\frac{2\kappa_0}{(p-1)(1-d^2)} -\cos \theta_3 (\cos \theta_4)^2...  (\cos \theta_m)^{m-2}\big|  \le  \frac{C \epsilon_1}{1-d^2}.
  \end{eqnarray*}
Since
$$\cos \theta_i \ge \frac{3}{4}$$
by hypothesis, we have the non-degeneracy of $\Phi$ (voir (\ref{172})) near the point $(\hat v,\hat d,\hat \theta)$ with $\hat v= R_{\hat {\theta}}(\kappa(\hat d,.)e_1,0)$. Applying the implicit function theorem, we conclude the proof of Proposition \ref{5.11}.
\end{proof}

\medskip

\subsection{Dynamics of $q$, $d$ and $\theta$}
Let us apply Proposition \ref{5.11} with $v=(w,\partial_s w) (s^*)$, $\hat d=d^*$ and $\hat \theta=0$. Clearly, from (\ref{56,5}) and (\ref{53,5}), we have $ ||\hat q||_{\mathcal{H}}\le \epsilon^*.$ Assuming that
$$\epsilon^*\le \epsilon_0$$
defined in Proposition \ref{5.11}, we see that the proposition applies, and from the continuity of $(w,\partial_s w)$ from $[s^*,\infty)$ to $\mathcal{H}$, we have a maximal $\bar s> s^*$, such that $(w(s),\partial_s w(s))$ can be modulated in the sense that
\begin{eqnarray}\label{ro}
\begin{pmatrix} w(y,s)\\\partial_s w(y,s) \end{pmatrix} =R_{\theta (s)}\left[\begin{pmatrix} \kappa(d(s),y)e_1\\0\end{pmatrix} +\begin{pmatrix} q_1(y,s)\\q_2(y,s)\end{pmatrix}\right],
\end{eqnarray}
where the parameters $d(s)\in (-1,1)$ and $\theta(s)=(\theta_2(s),..., \theta_m(s))$ are such that for all $s\in [s^*, \bar s)$
\begin{equation}\label{167,7}
\bar\pi_{0}^{d(s)}\begin{pmatrix} q_{1,1}(s)\\ q_{2,1}(s)\end{pmatrix} =0,\mbox{ and }\tilde\pi_{0}^{d(s)}\begin{pmatrix} q_{1,j}(s)\\ q_{2,j}(s)\end{pmatrix} =0,\; \forall j=2,..m,
\end{equation}
and
\begin{equation}\label{179}
\forall i=2,...,m\; \cos \theta_i(s)\ge \frac{1}{2} \mbox{ and }   ||q(s)||_{\mathcal{H}}\le \epsilon\equiv 2K_0 K_1 \epsilon^*,   
\end{equation}
where $K_1>0$ is defined in Proposition \ref{5.11} and $K_1>1$ is a constant that will be fixed below in (\ref{259}).

\medskip

Two cases then arise:\\
- Case 1: $\bar s=+\infty;$\\
- Case 2: $\bar s<+\infty;$ in this case, we have an equality case in (\ref{179}), i.e. $cos \theta_i(\bar s)= \frac{1}{2} $ for some $i=2,...,m$, or $ ||q(\bar s)||_{\mathcal{H}}=  2K_0 K_1 \epsilon^*$.\\
At this stage, we see that controlling the solution $(w(s),\partial_s w (s))\in \mathcal{H} $ is equivalent to controlling $q\in \mathcal{H}$, $d\in (-1,1)$ and $\theta(s)\in \mathbb{R}^{m-1}$.
\medskip

Before giving the dynamics of this parameters, we need to introduce some notations.\\
From (\ref{167,7}), we will expand $\bar q$ and $\tilde q$ respectively according to the spectrum of the linear operators $\bar L_d$ and $\tilde L_d$ as in (\ref{ma}) and (\ref{s}):
\begin{align}\label{180}
 \begin{pmatrix}  q_{1,1}(y,s)\\ q_{2,1}(y,s) \end{pmatrix}&=\alpha_{1,1} \bar F_1^d (y)+\begin{pmatrix}  q_{-,1,1}(y,s)\\ q_{-,2,1} (y,s)\end{pmatrix} \\
\label{180*}
\forall j \in \{1,...,m\},\; \begin{pmatrix} q_{1,j}(y,s)\\ q_{2,j} (y,s)\end{pmatrix}&=\begin{pmatrix}  q_{-,1,j}(y,s)\\ q_{-,2,j}(y,s) \end{pmatrix}
\end{align}
where
\begin{equation}\label{181}
\alpha_{1,1} = \bar\pi_1^{d(s)} \begin{pmatrix}  q_{1,1}\\ q_{2,1} \end{pmatrix},\;\alpha_{0,1} = \bar\pi_0^{d(s)} \begin{pmatrix}  q_{1,1}\\ q_{2,1} \end{pmatrix}=0,\;\alpha_{-,1}(s)=\sqrt{\bar\varphi_d (\begin{pmatrix}  q_{-,1,1}\\ q_{-,2,1} \end{pmatrix},\begin{pmatrix}  q_{-,1,1}\\ q_{-,2,1} \end{pmatrix})}
\end{equation}
\begin{equation}\label{181'}
\alpha_{0,j}= \tilde\pi_0^{d(s)}\begin{pmatrix} q_{1,j}\\ q_{2,j} \end{pmatrix}=0,\;\alpha_{-,j}(s)=\sqrt{\tilde\varphi_d (\begin{pmatrix}  q_{-,1,j}\\ q_{-,2,j} \end{pmatrix},\begin{pmatrix}  q_{-,1,j}\\ q_{-,2,j} \end{pmatrix})}
\end{equation}
and
\begin{equation*}
\begin{pmatrix}  q_{-,1,1}\\ q_{-,2,1} \end{pmatrix}=\bar\pi_{-}^{d}\begin{pmatrix} q_{1,1}\\q_{2,1} \end{pmatrix}
\end{equation*}
\begin{equation*}
\forall j \in \{1,...,m\},\;\begin{pmatrix}  q_{-,1,j}\\ q_{-,2,j} \end{pmatrix}=\tilde\pi_{-}^{d}\begin{pmatrix} q_{1,j}\\q_{2,j} \end{pmatrix}
\end{equation*}
From (\ref{180}), (\ref{180*}), (\ref{36,55}) (\ref{36,7}) and (\ref{50.5}), we see that for all $s \ge s_0$,
\begin{eqnarray}\label{183}
\notag
 \frac{1}{C_0} \alpha_{-,1}(s) &\le& ||\begin{pmatrix}  q_{-,1,1}\\ q_{-,2,1} \end{pmatrix}||_{\mathcal{H} } \le C_0  \alpha_{-,1}(s)\\ 
 \frac{1}{C_0}(| \alpha_{1,1}(s)|+ \alpha_{-,1}(s)) &\le& || \begin{pmatrix}  q_{1,1}\\ q_{2,1} \end{pmatrix}||_{\mathcal{H} } \le C_0(| \alpha_{1,1}(s)|+ \alpha_{-,1}(s)) \\
 \frac{1}{C_0} \alpha_{-,j}(s) &\le& || \begin{pmatrix}  q_{1,j}\\ q_{2,j} \end{pmatrix}||_{\mathcal{H} } \le C_0 \alpha_{-,j}(s) \notag
\end{eqnarray}
for some $C_0 > 0$. In the following proposition, we derive from (\ref{170}) and (\ref{170'}) differential inequalities satisfied by $ \alpha_{1,1}(s)$, $ \alpha_{-,1}(s)$, $ \alpha_{-,j}(s)$, $\theta_i(s)$ and $d(s)$.
Introducing
\begin{equation}\label{209}
R_{-}(s)=-\int_{-1}^{1} \mathcal{ F}_{d }(q_1) \rho dy,
\end{equation}
where 
\begin{equation}\label{*F*}
\mathcal{F}_{d(s)} (q_1 (y,s))=\frac{|\kappa(d ,\cdot)e_1+q_1|^{p+1}}{p+1}-\frac{\kappa(d ,\cdot)^{p+1}}{p+1}-\kappa(d ,\cdot)^p q_{1,1}-\frac{p}{2} \kappa(d ,\cdot)^{p-1}  q_{1,1}^2-\frac{\kappa(d ,\cdot)^{p-1}}{2} \sum_{j=2}^{m} q_{1,j}^2,
\end{equation}
we claim the following

\begin{prop}{\bf (Dynamics of the parameters)}\label{5.2}
For $\epsilon^*$ small enough and for all $s\in [s^*,\bar s)$, we have:

\item{(i)} {\bf (Control of the modulation parameter)}
\begin{equation}\label{184}
\sum_{i=2}^{m}|\theta_i'|+\frac{|d'|}{1-d^2}\le C_0 || q||_{\mathcal{H}}^2.
\end{equation}
\item{(ii)} {\bf (Projection of equation (\ref{170}) on the different eigenspaces of $\bar L_d$ and $\tilde L_d$)}
\begin{align}\label{185}
| \alpha_{1,1}'(s) - \alpha_{1,1}(s)|&\le C_0 || q||_{H}^2.
\\\label{186}
\left( R_{-}+\frac{1}{2}( \alpha_{1,1}^2+ \alpha_{-,j}^2)\right)'&\le -\frac{4}{p-1}\int_{-1}^{1}(q_{-,2,1}^2+ q_{-,2,j}^2)\frac{\rho}{1-y^2}dy+ C_0 ||q(s)||_{\mathcal{H}}^{3},
\end{align}
for $j\in \{2,...,m\}$ and $R_{-}(s)$ defined in $(\ref{209})$, satisfying
\begin{equation}\label{187}
|R_{-}(s)|\le C_0 ||q(s)||_{\mathcal{H}}^{1+\bar p} \mbox{ where } \bar p=\min(p,2)>1.
\end{equation}
\item{(iii)} {\bf (An additional relation)}
\begin{equation}\label{188}
\frac{d}{ds}\int_{-1}^{1}  q_{1,1} q_{2,1} \rho \le -\frac{4}{5}\bar \alpha_{-,1}^2+ C_0\int_{-1}^{1} q_{-,2,1}^2\frac{\rho}{1-y^2}+ C_0  ||q(s)||_{\mathcal{H}}^{2}.
\end{equation}
For $j\in \{2,...,m\},$ we have:
\begin{equation}\label{188'}
\frac{d}{ds}\int_{-1}^{1}  q_{1,j}  q_{1,j} \rho \le -\frac{4}{5}\tilde \alpha_{-,j}^2+ C_0\int_{-1}^{1}  q_{2,j}^2\frac{\rho}{1-y^2}+ C_0  ||q(s)||_{\mathcal{H}}^{2}.
\end{equation}
\item{(iv)} {\bf (Energy barrier)}
\begin{equation}
 \alpha_{1,1} (s)\le C_0\alpha_{-,1}(s)+ C_1  \sum_{j=2}^{m} \alpha_{-,j}(s).\label{189}
\end{equation}
\end{prop}

\begin{proof}\label{chdg}
The proof follows the general framework developed by Merle and Zaag in the real case (see Proposition $5.2$ in \cite{MR2362418}), then adapted to the complex-valued case in \cite{<3} 5(see Proposition $4.2$ page $5915$ in \cite{<3}). However, new ideas are needed, mainly because we have $(m-1)$ rotation parameters in the modulation technique (see Proposition \ref{5.11} above), rather than only one in the complex-valued case. For that reason, in the following, we give details only for the "new" terms, referring the reader to the earlier literature for the "old" terms.

Let us first write an equation satisfied by $q$ defined in (\ref{ro}). We put the equation (\ref{equa}) satisfied by $w$ in vectorial form:
\begin{eqnarray}\label{l1}
\begin{array}{l}
 \partial_s w_1=w_2\\
\partial_s w_2=\mathcal{L}w_1-\frac{2(p+1)}{(p-1)^2}w_1+|w_1|^{p-1}w_1-\frac{p+3}{p-1} w_2- 2 y \partial_y w_2.
\end{array}
\end{eqnarray}
 We replace all the terms of (\ref{l1}) by their expressions from (\ref{ro}). Precisely, for the terms of the right hand side of (\ref{l1})  we have: 
\begin{eqnarray}\label{l2}
\begin{array}{l}
\partial_s w_1=R_{\theta }\left(d'\partial_d \kappa e_1+\partial_s q_1\right)+\sum_{i=1}^{m}\theta_i'\frac{\partial R_\theta}{\partial \theta_j}\left( \kappa_d e_1+q_1\right),\\
\partial_s w_2=R_{\theta }\left(\partial_s q_1\right)+\sum_{i=1}^{m}\theta_i'\frac{\partial R_\theta}{\partial \theta_j}\left(q_2\right). \notag
\end{array}\end{eqnarray}
For the terms on the left hand side of (\ref{l1})  we have:

$$w_2=R_{\theta }(q_2), \,\mathcal{L}w_1=R_{\theta }( \mathcal{L}(\kappa_d e_1)+\mathcal{L} q_1),\,|w_1|=|\kappa_d e_1+ q_1|,\,\partial_y w_2=R_{\theta }\partial_y q_2.$$
Then, multiplying by $R_{\theta}^{-1}$, using the fact that $(\kappa(d ,\cdot),0)$ is a stationnary solution and dissociating the first and $j$th component of these equations, we get for all $s\in [s^*,\bar s)$, for all $j \in \{2,...,m\}$:

 \begin{align}\label{170}
 \frac{\partial }{\partial s}\begin{pmatrix}  q_{1,1}\\ q_{2,1}
 \end{pmatrix}&=\bar L_{d(s)} \begin{pmatrix}  q_{1,1}\\ q_{2,1} \end{pmatrix}+\begin{pmatrix}0\\ {f}_{d(s),1}(q_1)\end{pmatrix}-d'(s)\begin{pmatrix}\partial_d \kappa(d,y)\\ 0\end{pmatrix}-\sum_{i=2}^{m}\theta_i'(s)\begin{pmatrix}  a_{i,1,1} \\ a_{i,2,1} \end{pmatrix},
 \\
\frac{\partial }{\partial s} \begin{pmatrix}  q_{1,j}\\q_{2,j}
 \end{pmatrix}&=\tilde L_{d(s)}\begin{pmatrix} q_{1,j}\\ q_{2,j} \end{pmatrix}+\begin{pmatrix} 0\\{f}_{d(s),j}(q_1) \end{pmatrix}-\sum_{i=2}^{m}\theta_i'(s)\begin{pmatrix}  a_{i,1,j} \\ a_{i,2,j} \end{pmatrix},
 \label{170'}\end{align}
 where $\bar L_{d(s)} ,\tilde L_{d(s)}, f_{d(s),1}$ and $f_{d(s),j}$ are defined in (\ref{barL_d}), (\ref{tildeL_d}), (\ref{barf_d}) and (\ref{tildef_d}), and $a_i$ by

 \begin{eqnarray}\label{a_i}
 \begin{pmatrix} a_{i,1}\\a_{i,2}\end{pmatrix}=\begin{pmatrix}A_i (\kappa_d e_1+q_1)\\A_i (q_2) \end{pmatrix},
\end{eqnarray}
 with $a_{i,1}=(a_{i,1,1},a_{i,1,2},...,a_{i,1,m})\in \mathbb{R}^{m}$ and $a_{i,2}=(a_{i,2,1},a_{i,2,2},...,a_{i,2,m})\in \mathbb{R}^{m}$.

Let $i\in \{2,\cdots,m\}$, Projecting equation (\ref{170}) with the projector $\bar \pi_\lambda^d$ (\ref{barpi}) for $\lambda=0$ and $\lambda=1$, we write

\begin{eqnarray}\label{d0}
\bar \pi_\lambda^d (\partial_s \begin{pmatrix}  q_{1,1}\\ q_{2,1}
 \end{pmatrix})
&=&\bar \pi_\lambda^d (\bar L_{d(s)}  \begin{pmatrix}  q_{1,1}\\ q_{2,1} \end{pmatrix})+\bar \pi_\lambda^d \begin{pmatrix}0\\  {f}_{d(s),1}(q_1) \end{pmatrix}-d'(s)\bar \pi_\lambda^d \begin{pmatrix}\partial_d \kappa(d,y)\\ 0\end{pmatrix}\notag\\
&-&\sum_{i=2}^{m}\theta_i'(s)\bar \pi_\lambda^d\begin{pmatrix}  a_{i,1,1} \\ a_{i,2,1} \end{pmatrix}.
\end{eqnarray}

Note that, expect the last term, all the terms of (\ref{d0}) can be controled exactly like the real case using (\ref{179}) (for details see page 105 in \cite{MR2362418}). So, we recall that we have:\\
 \begin{eqnarray}\label{d1}
|\bar \pi_\lambda^d( \partial_s \begin{pmatrix} q_{1,1}\\q_{2,1} \end{pmatrix})-\alpha_{\lambda,1}'|\le \frac{C_0}{1-d^2} |d'||| q||_{\mathcal{ H}},
\end{eqnarray}
 \begin{eqnarray}\label{d2}
\bar \pi_\lambda^d( L_d \begin{pmatrix} q_{1,1}\\q_{2,1} \end{pmatrix})=
\lambda \alpha_{\lambda,1},
\end{eqnarray}
 \begin{eqnarray}\label{d3}
|\bar \pi_\lambda^d \begin{pmatrix}0\\  {f}_{d(s),1}(q_1) \end{pmatrix}|\le C_0|| q||_{\mathcal{ H}}^2,
\end{eqnarray}
 \begin{eqnarray}\label{d4}
|\bar \pi_\lambda^d \begin{pmatrix}\partial_d \kappa(d,y)\\ 0\end{pmatrix}|=-\frac{2\kappa_0}{(p-1)(1-d^2)}    \bar \pi_\lambda^d (F_0^d)       =-\frac{2\kappa_0}{(p-1)(1-d^2)}\delta_{\lambda,0}.
\end{eqnarray}
Now, we focus on the study of the last term of (\ref{d0}). From the definition of $a_i$ (\ref{a_i}) and $i)$ of Lemma \ref{prod}, we have:
 \begin{eqnarray*}\label{}
 \begin{pmatrix} a_{i,1,1}\\a_{i,2,1}\end{pmatrix}=\kappa_d \begin{pmatrix}<e_1,A_i (e_1)>\\0 \end{pmatrix}+\begin{pmatrix}<e_1,A_i (q_1)>\\<e_1,A_i (q_2)> \end{pmatrix}=\begin{pmatrix}<e_1,A_i (q_1)>\\<e_1,A_i (q_2)> \end{pmatrix}.
\end{eqnarray*}
Applying the projector $\bar \pi_\lambda^d$ (\ref{barpi}), we get
\begin{eqnarray}\label{juil}
|\bar \pi_\lambda^d \begin{pmatrix} a_{i,1,1}\\a_{i,2,1} \end{pmatrix}|&=&|\bar \pi_\lambda^d \begin{pmatrix}<e_1,A_i (q_1)>\\<e_1,A_i (q_2)> \end{pmatrix}|\le C ||   \begin{pmatrix}<e_1,A_i (q_1)>\\<e_1,A_i (q_2)> \end{pmatrix}  ||_{\mathcal{ H}}\notag\\
&\le& C (|| <e_1,A_i (q_1)>||_{\mathcal{ H}_0}+|| <e_1,A_i (q_2)>||_{ L_\rho^2}).
\end{eqnarray}
Using $ii)$ of Lemma \ref{prod}, we have:
\begin{eqnarray}\label{juil1}
|| <e_1,A_i (q_2)>||^2_{ L_\rho^2}=\int_{-1}^{1} <e_1,A_i (q_2)>^2\rho dy\le \int_{-1}^{1} |A_i (q_2)|^2\rho dy\le \int_{-1}^{1} |q_2|^2\rho dy,
\end{eqnarray}
and by the same way, using $ii)$ of Lemma \ref{prod} and the definition of $\mathcal{ H}_0$ (\ref{10}), we have
\begin{eqnarray}\label{juil2}
|| <e_1,A_i (q_1)>||^2_{\mathcal{ H}_0}&=&\int_{-1}^{1} <e_1,A_i (q_1)>^2\rho dy+\int_{-1}^{1} (<e_1,A_i (\partial_y q_1)>)^2(1-y^2)\rho dy\notag\\&\le&
 \int_{-1}^{1}( |q_1|^2+(1-y^2)|\partial_y q_1|^2)\rho dy.
\end{eqnarray}
From (\ref{juil}),  (\ref{juil1}) and  (\ref{juil2}), we have
\begin{eqnarray}\label{d5}
|\bar \pi_\lambda^d \begin{pmatrix} a_{i,1,1}\\a_{i,2,1} \end{pmatrix}|\le C_0 || q||_{\mathcal{ H}}.\end{eqnarray}
Using (\ref{d1}), (\ref{d1}), (\ref{d2}), (\ref{d3}), (\ref{d4}), (\ref{d5}), and the fact that $\alpha_{0,1}\equiv \alpha_{0,1}'\equiv 0$ (see (\ref{181})), we get for $\lambda=0,1$:
\begin{align}\label{A1}
\frac{2 \kappa_0}{(p-1)(1-d^2)} |d'|&\le \frac{C_0}{1-d^2}|d'| || q||_{\mathcal{ H}}+C_0 || q||_{\mathcal{ H}}^2+C_0 || q||_{\mathcal{ H}}\sum_{j=2}^{m}|\theta_i'|
\\\label{A2}
|\bar\alpha_1'(s)-\bar\alpha_1(s)|&\le\frac{C_0}{1-d^2}|d'| || q||_{\mathcal{ H}} +C_0 || q||_{\mathcal{ H}}^2+C_0|| q||_{\mathcal{ H}}\sum_{j=2}^{m}|\theta_i'|.
\end{align}
Now, projecting equation (\ref{170'}) with the projector $\tilde \pi_0^d$ (\ref{pi}), where $j\in \{2,\cdots,m\}$, we get:
\begin{eqnarray}\label{98,5}
\tilde \pi_0^d(\partial_s\begin{pmatrix}  q_{1,j}\\q_{2,j}
 \end{pmatrix} )=\tilde \pi_0^d( \tilde L_{d(s)}\begin{pmatrix} q_{1,j}\\ q_{2,j} \end{pmatrix})+\tilde \pi_0^d \begin{pmatrix} 0\\f_{d(s),j}  (q_1) \end{pmatrix}-\sum_{i=2}^{m}\theta_i'(s) \tilde \pi_0^d\begin{pmatrix}  a_{i,1,j} \\ a_{i,2,j} \end{pmatrix}.
\end{eqnarray}

From the complex-valued case we recall that we have (for details see page $5917$ in \cite{<3}, together with Lemma \ref{prod}):
\begin{equation}\label{B1}
|\tilde \pi_0^d  (\partial_s \begin{pmatrix} q_{1,j}\\ q_{2,j} \end{pmatrix})|\le \frac{C_0}{1-d^2} |d '| || q||_{\mathcal{ H}},
\end{equation}
\begin{equation}\label{B2}
\tilde \pi_0^d  (\tilde L_{d }\begin{pmatrix} q_{1,j}\\ q_{2,j} \end{pmatrix})=0,
\end{equation}
\begin{equation}\label{B3}
\Big{|}\tilde \pi_0^d  \begin{pmatrix} 0\\\tilde{f}_{d(s)}(q_1) \end{pmatrix}\Big{|} \le C_0  || q||_{\mathcal{ H}}^2,
\end{equation}

\begin{equation}\label{B4}
\tilde \pi_0^d\begin{pmatrix}   \kappa_d \\0 \end{pmatrix}   =1.
\end{equation}

Thus, only the last term in (\ref{98,5}) remains to be treated in the following.

From the definition of $a_i$ (\ref{a_i}), we recall that
 \begin{eqnarray}\label{}
 \begin{pmatrix} a_{i,1,j}\\a_{i,2,j}\end{pmatrix}= \begin{pmatrix}<e_j,A_i (e_1)>\kappa_d\\0 \end{pmatrix}+\begin{pmatrix}<e_j,A_i (q_1)>\\<e_j,A_i (q_2)>  \end{pmatrix}.
\end{eqnarray}
By $i)$ of Lemma \ref{prod}:
 \begin{eqnarray}\label{q1}
\sum_{i=2}^{m}\theta_i'(s) \begin{pmatrix} a_{i,1,j}\\a_{i,2,j}\end{pmatrix}=\theta_j'(s)\begin{pmatrix}   (\prod\limits_{l =j+1}^{m} \cos\theta_l )  \kappa_d  \\0 \end{pmatrix}+\sum_{i=2}^{m}\theta_i'(s) \begin{pmatrix}<e_j,A_i (q_1)>\\<e_j,A_i (q_2)>  \end{pmatrix},
\end{eqnarray}
where by convention $\prod\limits_{l =m+1}^{m} \cos\theta_l =1$ if $j=m$.

Applying the projection $\tilde \pi_0^d$ to (\ref{q1}) and using (\ref{B4}), we see that
 \begin{eqnarray}\label{100'}
 \Big{|} \sum_{i=2}^{m}\theta_i'(s) \tilde \pi_0^d \begin{pmatrix} a_{i,1,j}\\a_{i,2,j}\end{pmatrix}-\theta_j'(s) \prod\limits_{l =j+1}^{m} \cos\theta_l    \Big{|} &\le& \sum_{i=2}^{m}|\theta_i'(s)
 | \Big{|} \tilde \pi_0^d \begin{pmatrix} <e_j,A_i (q_1)>\\<e_j,A_i (q_2)> \end{pmatrix}  \Big{|}\notag\\&\le& C_0  || q||_{\mathcal{ H}}\sum_{i=2}^{m}|\theta_i'(s)|,
\end{eqnarray}
where, we use the fact that

\begin{eqnarray}\label{B44}
|\tilde \pi_0^d\begin{pmatrix}  <e_j,A_i (q_1)>\\<e_j,A_i (q_2)> \end{pmatrix}|\le C_0 || q||_{\mathcal{ H}},
\end{eqnarray}
which follows by the same techniques as in (\ref{juil},) (\ref{juil1}) and (\ref{juil2}).   

Using  (\ref{98,5}),(\ref{B1}), (\ref{B2}), (\ref{B3}) and (\ref{100'}), and recalling from (\ref{179}) that
$$\prod\limits_{l =j+1}^{m} \cos\theta_l \ge (\frac{1}{2})^{m-j},$$

we get for any $j\in  \{2,...,m\}$:

 \begin{eqnarray}\label{ll}
 \Big{|} \theta_j'(s) \Big{|} \le  \frac{C_0}{1-d^2}|d'| || q||_{\mathcal{ H}}+C_0 || q||_{\mathcal{ H}}^2+C_0|| q||_{\mathcal{ H}}\sum_{i=2}^{m}| \theta_i'|  .\end{eqnarray}

Using (\ref{ll}) together with (\ref{A1}), we see that
 \begin{eqnarray*}\label{}
\sum_{j=2}^{m} \Big{|} \theta_j'(s) \Big{|}+  \frac{|d'|}{1-d^2} \le C_0 \frac{|d'|}{1-d^2} || q||_{\mathcal{ H}}+C_0 || q||_{\mathcal{ H}}^2+C_0| || q||_{\mathcal{ H}}\sum_{i=2}^{m}| \theta_i'| ,
\end{eqnarray*}
Thus, using again (\ref{179}) and taking $\epsilon$ small enough, we get
\begin{eqnarray*}\label{}
\sum_{j=2}^{m} \Big{|} \theta_j'(s) \Big{|}+ \frac{|d'|}{1-d^2} \le C_0 || q||_{\mathcal{ H}}^2,
\end{eqnarray*}
which yields (\ref{184}). Then, using (\ref{A2}) together with (\ref{184}) gives (\ref{185}).

 \medskip

 For estimations (\ref{187})  (\ref{188})  (\ref{188'}) (\ref{189}), the study in the complex case (Subsection $4.3$ page $5914$ in \cite{<3}) can be adapted without any difficulty to the vector-valued case. For the reader convenience, we detail for example the energy barrier (\ref{189}):

 Using the definition of $q(y,s)$ (\ref{168}), we can make an expansion of $E(w(s),\partial_s w(s))$ (\ref{15}) for $q\rightarrow 0$ in $\mathcal{H}$ and get after from straightforward computations:
 \begin{equation}
  E(w(s),\partial_s w(s))= E(\kappa_0,0)+\frac{1}{2}\left[\bar \varphi_d(\begin{pmatrix}  q_{1,1}\\ q_{2,1} \end{pmatrix}, \begin{pmatrix}  q_{1,1}\\ q_{2,1} \end{pmatrix})+\sum_{i=2}^{m}\tilde \varphi_d(\begin{pmatrix}  q_{1,j}\\ q_{2,j} \end{pmatrix},\begin{pmatrix}  q_{1,j}\\ q_{2,j} \end{pmatrix})\right]-\int_{-1}^{1}\mathcal{F}_d (q_1) \rho dy \label{228}
 \end{equation}
 where $ \bar \varphi_d$, $ \tilde \varphi_d$ and $\mathcal{F}_d (q_1)$ are defined in (\ref{134'}), (\ref{134}) and (\ref{*F*}).\\
 
Using the argument in the real case (see page 113 in \cite{MR2362418}) we see that for some $C_0,C_1 >0$ we have:
 \begin{equation}\label{230}
  \bar \varphi_d(\begin{pmatrix}  q_{1,1}\\ q_{2,1} \end{pmatrix}, \begin{pmatrix}  q_{1,1}\\ q_{2,1} \end{pmatrix}) \le C_0 \alpha_{1,1}^2- C_1\alpha_{-,1}^2.
 \end{equation}
From (\ref{180*}), (\ref{181'}) and (\ref{50.5}), we see by definition that
 \begin{equation}\label{105'}
0\le   \tilde\varphi_d(\begin{pmatrix}  q_{1,j}\\ q_{2,j} \end{pmatrix}, \begin{pmatrix}  q_{1,j}\\ q_{2,j} \end{pmatrix})=\alpha_{-,j}^2.
 \end{equation}

 Since we have from (\ref{209}), (\ref{187}), (\ref{179}), (\ref{105'}) and (\ref{183}):
 \begin{equation}
  \left | \int_{-1}^{1}\mathcal{F}_d (q_1) \rho dy \right | \le C ||q(s)||_{\mathcal{H}}^{\bar p+1}\le C \epsilon^{\bar p-1} ( \alpha_{1,1}^2+ \alpha_{-,1}^2+\sum_{i=2}^{m} \alpha_{-,j}^2),\label{229}
 \end{equation}

Using (\ref{17}), (\ref{228}), (\ref{230}) and (\ref{229}), we see that taking $\epsilon$ small enough so that $C \epsilon^{\bar p-1}\le \frac{C_1}{4}$, we get
 \begin{equation*}
  0\le E(w(s),\partial_s w(s))-E(\kappa_0,0)\le \left( \frac{C_0}{2}+\frac{C_1}{4}\right)  \alpha_{-,1}^2-\frac{C_1}{4}\alpha_{1,1}^2+\left(\frac{1}{2}+\frac{C_1}{4}\right)\sum_{i=2}^{m} \alpha_{-,j}^2.
 \end{equation*}
 which yields (\ref{189}).
 \end{proof}
 
 \subsection{Exponential decay of the different components}\label{4.4}
Our aim is to show that $||q(s)||_{\mathcal{H}}\rightarrow 0$ and that both $\theta$ and $d$ converge as $s\rightarrow \infty$. An important issue will be to show that the unstable mode $\alpha_{1,1}$, which satisfies equation (\ref{183}) never dominates. This is true thanks to item $(iv)$ in Proposition \ref{5.2}.\\

If we introduce
\begin{equation}\label{234}
 \lambda(s)=\frac{1}{2} \log\left(\frac{1+d(s)}{1-d(s)}\right), a(s)= \alpha_{1,1}(s)^2\, \mbox{and}\, b(s)=\alpha_{-,1}(s)^2+\sum_{j=2}^{m}\alpha_{-,j}(s)^2+R_-(s) 
\end{equation}
(note that $d(s)=\tanh(\lambda(s))$), then we see from (\ref{187}), (\ref{183}) and (\ref{179}) that for all $s\in [s^*, \bar s)$           

 $|R_-(s)|=|b(s)-(\alpha_{-,1}(s)^2+\sum_{j=2}^{m}\alpha_{-,j}(s)^2)|\le C_0 \epsilon^{\bar p-1}(\alpha_{1,1}(s)^2+\alpha_{-,1}(s)^2+\sum_{j=2}^{m}\alpha_{-,j}(s)^2)$, hence
\begin{equation}\label{235}
 \frac{99}{100}\alpha_{-,1}(s)^2+ \frac{99}{100}\sum_{j=2}^{m}\alpha_{-,j}(s)^2-\frac{1}{100} a\le b\le  \frac{101}{100}\alpha_{-,1}(s)^2+ \frac{101}{100}\sum_{j=2}^{m}\alpha_{-,j}(s)^2+\frac{1}{100} a
\end{equation}
for $\epsilon$ small enough. Therefore, using Proposition \ref{5.2}, estimate (\ref{179}), (\ref{183}) and the fact that $\lambda'(s)=\frac{d'(s)}{1-d(s)^2}$, we derive the following:
\begin{cl}\label{5.5}{\bf (Relations between $a$, $b$, $\lambda$, $\theta$, $\int_{-1}^1 q_{1,1} q_{2,1} \rho$ and $\int_{-1}^1 q_{1,j}q_{2,j} \rho$)} There exist positive $\epsilon_4$, $K_4$ and $K_5$ such that if $\epsilon^*\le \epsilon_4$, then we have for all $s\in[s^*,\bar s]$ and $j=2,...,m$:
\item{(i)} {\bf (Size of the solution)}
\begin{align}
\frac{1}{K_4}(a(s)+b(s))\le || q(s)||_\mathcal{H}^2&\le K_4 (a(s)+b(s))\le K_4^2 \epsilon^2,\label{236}
\\
 |\theta'(s)|+|\lambda'(s)|&\le K_4 (a(s)+b(s))\le K_4^2 ||q(s)||_\mathcal{H}^2 ,\label{237}
\\
 \left | \int_{-1}^{1}q_{1,1}q_{1,1}\rho \right |&\le K_4 (a(s)+b(s)),\label{238}
\\
\left | \int_{-1}^{1}q_{1,j}q_{1,j}\rho \right |&\le K_4 b(s),\label{238'}
 \end{align}
and (\ref{235}) holds.
\item{(ii)} {\bf (Equations)}
\begin{align}\label{239}
 \frac{3}{2} a-K_4 \epsilon b&\le a' \le \frac{5}{2} a-K_4 \epsilon b,
\\\label{240}
b'&\le -\frac{8}{p-1}\int_{-1}^{1}(q_{-,2,1}^2+q_{-,2,j}^2)\frac{\rho}{1-y^2}dy+ K_4\epsilon (a+b),
\\\notag
 \frac{d}{ds}\int_{-1}^1 (q_{1,1}  q_{2,1}+ q_{1,j} q_{2,j}) \rho &\le -\frac{3}{5} b+K_4 \int_{-1}^1 (q_{-,2,1}^2+ q_{2,j}^2)\frac{\rho}{1-y^2}+K_4 a.
  \end{align}
\item{(iii)} {\bf (Energy barrier)} we have
   \begin{equation*}
    a(s)\le K_5 b(s).
   \end{equation*}
\end{cl}

\begin{proof}[End of the Proof of Theorem \ref{theo3}]
 Now, we are ready to finish the proof of Theorem \ref{theo3} just started at the beginning of Section \ref{sec,mod}. Let us define $s_2^* \in [s^*,\bar s]$ as the first $s \in [s^*,\bar s]$ such that
      \begin{equation}\label{247}
       a(s)\ge \frac{b(s)}{5 K_4}
      \end{equation}
where $K_4$ is introduced in Corollary \ref{5.5}, or $s^*_2=\bar s$ if (\ref{247}) is never satisfied on $[s^*,\bar s]$. We claim the following:
\begin{cl}\label{5.6}
 There exist positive $\epsilon_6$, $\mu_6$, $K_6$ and $f\in C^1([s^*, s^*_2]$ such that if $\epsilon\le \epsilon_6$, then for all $s\in [s^*,s_2^*]$: \\
(i)
\begin{equation*}
 \frac{1}{2}f(s)\le b(s)\le 2 f(s)\mbox{ and }f'(s)\le -2\mu_6f(s),
\end{equation*}
(ii)
\begin{equation*}
|| q(s)||_\mathcal{H}\le K_6 || q(s^*)||_\mathcal{H}e^{-\mu_6(s-s^*)}\le K_6 K_1 \epsilon^*e^{-\mu_6(s-s^*)}.
\end{equation*}
\end{cl}
\begin{proof} The proof of Claim 5.6 page 115 in \cite{MR2362418} remains valid where $f(s)$ is given by 
 $$f(s)=b(s)+\eta_6\int_{-1}^1 (q_{1,1} q_{2,1}+ \sum_{j=2}^{m}q_{1,j} q_{2,j}) \rho,$$
where $\eta_6 >0$ is fixed small independent of $\epsilon$. 
\end{proof}
\begin{cl}\label{5.7}
 (i) There exists $\epsilon_7>0$ such that for all $\sigma>0$, there exists $K_7(\sigma)>0$ such that if $\epsilon \le \epsilon_7$, then
\begin{equation*}
 \forall s \in [s_2^*, \min (s_2^*+\sigma, \bar s)],\,|| q(s)||_\mathcal{H}\le K_7 || q(s^*)||_\mathcal{H}e^{-\mu_6(s-s^*)}\le K_7 K_1 \epsilon^*e^{-\mu_6(s-s^*)}
\end{equation*}
and
\begin{equation*}
 |\theta_i(s)|\le C\frac{(K_7 K_1 \epsilon^*)^2}{2\mu_6}
\end{equation*}

where $\mu_6$ has been introduced in Claim \ref{5.6}.\\
(ii) There exists $\epsilon_8>0$ such that if $\epsilon\le \epsilon_8$, then
\begin{equation}\label{254}
  \forall s \in (s_2^*,  \bar s],\; b(s)\le a(s) \left( 5 K_4 e^{-\frac{(s-s_2^*)}{2}}+\frac{1}{4 K_5}\right)
\end{equation}
where $K_4$ and $K_5$ have been introduced in Corollary \ref{5.5}.
\end{cl}
\begin{proof}
 The proof is the same as the proof of Claim 5.7 page 117 in \cite{MR2362418}.
\end{proof}
Now, in order to conclude the proof of Theorem \ref{theo3}, we fix $\sigma_0>0$ such that 
\begin{equation*}
 5K_4^{-\frac{\sigma_0}{2}}+\frac{1}{4K_5}\le \frac{1}{2K_5},
\end{equation*}
where $K_4$ and $K_5$ are introduced in Claim \ref{5.5}. Then, we fix the value of
\begin{equation}\label{259}
K_0=\max(2,K_6,K_7(\sigma_0)),
\end{equation}
and the constants are defined in Claims \ref{5.6} and \ref{5.7}. Then, we fix
\begin{equation*}
 \epsilon_0=\min \left(1,\epsilon_1,\frac{\epsilon_i}{2K_0K_1}\mbox{ for }i\in\{4,6,7,8\}\right)
\end{equation*}
and the constants are defined in Claims \ref{5.5}, \ref{5.6} and \ref{5.7}.
Now, if $\epsilon^*\le \epsilon_0$, then Claim \ref{5.5}, Claim \ref{5.6} and Claim \ref{5.7} apply. We claim that for all $s\in[s^*,\bar s]$,
 \begin{eqnarray}\label{261}
 || q(s)||_\mathcal{H}\le K_0 || q(s^*)||_\mathcal{H}e^{-\mu_6(s-s^*)}\le K_0 K_1 \epsilon^*e^{-\mu_6(s-s^*)}=\frac{\epsilon}{2}e^{-\mu_6(s-s^*)}.
 \end{eqnarray}
Indeed, if $s\in[s^*,\min(s_2^*+\sigma_0,\bar s)]$, then, this comes from $(ii)$ of Claim \ref{5.6} or $(i)$ of Claim \ref{5.7} and the definition of $K_0$ (\ref{259}).\\
Now, if $s_2^*+\sigma_0<\bar s$ and $s\in[s_2^*+\sigma_0,\bar s]$, then we have from (\ref{254}) and the definition of $\sigma_0$, $b(s)\le\frac{a(s)}{2K_5 }$ on the one hand. On the other hand, from $(iii)$ in Claim \ref{5.5}, we have $a(s)\le K_5 b(s)$, hence, $a(s)=b(s)=0$ and from (\ref{236}), $q(y,s)\equiv 0$, hence (\ref{261}) is satisfied trivially.\\
In particular, we have for all $s\in [s ^*,\bar s],\, ||q||_\mathcal{H}\le \frac{\epsilon}{2}$ and $ \cos \theta_i\ge 1-C \frac{\epsilon^2}{\mu_6^2}\ge\frac{3}{4}$, hence, by definition of $\bar s$ given right before (\ref{ro}), this means that $\bar s=\infty$.\\
From $(i)$ of Claim \ref{5.7} and (\ref{237}), we have
\begin{equation}\label{262}
 \forall s \ge s^*,||q(s)||_\mathcal{H}\le \frac{\epsilon}{2} e^{-\mu_6(s-s^*)}\mbox{ and }|\theta'(s)|+|\lambda'(s)|\le K_4^2 \frac{\epsilon^2}{4} e^{-2\mu_6(s-s^*)},
\end{equation}
where, $\theta(s)=(\theta_2 (s),...,\theta_m (s))$.

Hence, there is $\theta_\infty\in\mathbb{R}^{m-1}$ , $\lambda_\infty$ in $\mathbb{R}$ such that $\theta(s) \rightarrow \theta_\infty$, $\lambda(s) \rightarrow \lambda_\infty$ as $s\rightarrow \infty$ and
\begin{equation}\label{263}
  \forall s \ge s^*,| \lambda_\infty-\lambda(s) |\le C_1\epsilon^{*2} e^{-2\mu_6(s-s^*)}=C_2\epsilon^{2} e^{-2\mu_6(s-s^*)}
\end{equation}
\begin{equation}\label{263'}
  \forall s \ge s^*,| \theta_\infty-\theta(s) |\le C_1\epsilon^{*2} e^{-2\mu_6(s-s^*)}=C_2\epsilon^{2} e^{-2\mu_6(s-s^*)}
\end{equation}
for some positive $C_1$ and $C_2$. Taking $s=s^*$ here, we see that 

\begin{equation*}
| \lambda_\infty-\lambda^* |+|\theta_\infty |\le C_0\epsilon^* ,
\end{equation*}
where $\Omega=R_{\theta_\infty }(e_1)$. If $d_\infty=\tanh \lambda_\infty,$ then we see that $|d_\infty-d^*|\le C_3 (1-d^{*2})\epsilon^*.$\\
Using the definition of $q$ (\ref{ro}), (\ref{262}), (\ref{263}) and (\ref{263'}) 
we write
\begin{eqnarray}\label{23,4}
&\,&\notag \Bigg|\Bigg|\begin{pmatrix}w(s)\\\partial_s w(s)\end{pmatrix}-\begin{pmatrix}\kappa(d _\infty,\cdot)\Omega_\infty\\0\end{pmatrix}\Bigg|\Bigg|_\mathcal{H}\\
&\le&\notag \Bigg|\Bigg|\begin{pmatrix}w(s)\\\partial_s w(s)\end{pmatrix}-\begin{pmatrix}\kappa(d (s),\cdot)\Omega_\infty\\0\end{pmatrix}\Bigg|\Bigg|_\mathcal{H}+
||(\kappa(d (s),\cdot)-\kappa(d _\infty,\cdot)) \Omega_\infty||_{\mathcal{H}_0}\\&+&\notag||\kappa(d_\infty,\cdot)||_{\mathcal{H}_0} |R_{\theta(s)}(e_1)-R_{\theta_{\infty }}(e_1)|\\
&\le& ||q(s)||_\mathcal{H}+C|\lambda_\infty-\lambda(s)|+C|\theta_\infty-\theta(s)|\le C_4 \epsilon^*e^{-\mu_6(s-s^*)},
\end{eqnarray}
where, we used the fact that $\theta\in \mathbb{R}^{m-1}\mapsto \mathcal{O}^m$ is a Lipschitz function (see (\ref{jac}) to be convinced) and $\lambda\in \mathbb{R}\mapsto \kappa(d,\cdot)\in \mathcal{H}_0$ is also Lipschitz, where $d=\tanh \lambda$ (see \ref{avion}).This concludes the proof of Theorem \ref{theo3} in the case where $\Omega^*=e_1$ (see (\ref{53,5})).
From rotataion invariance of equation (\ref{equa}), this yields the conclusion of Theorem \ref{theo3} in the general case.
\end{proof}
\medskip
\appendix

\section{A some technical estimates}\label{AA}

In this section, we give the proof of estimate (\ref{61,5}) and Lemma \ref{prod}.\\

\bigskip

\noindent{\bf \it Proof of estimate (\ref{61,5})}:\\
Using (\ref{R_theta}), we see that
\begin{equation}\label{60,5}
\frac{\partial R_{\theta}}{\partial \theta_j}= R_2 ... R_{j-1}\frac{\partial R_j}{\partial \theta_j}R_{j+1}...R_m.
\end{equation}
From (\ref{matrice}), we see that
\begin{equation}\label{60,5}
\frac{\partial R_{\theta}}{\partial \theta_j}= \Pi_j\circ R_{j}(\theta_j+\frac{\pi}{2})= R_{j}(\theta_j+\frac{\pi}{2})\circ  \Pi_j,
\end{equation}
where $\Pi_j$ is the orthogonal projection on the plane spanned by $e_1$ and $e_j$, and the rotation $ R_{j}(\alpha)$ is given by considering the matrix of $R_j$ defined in (\ref{matrice}), and changing $\theta_j$ into $\alpha$.

Since $$\partial R_{j}^{-1 }\partial R_{j}(\theta_j+\frac{\pi}{2})=\partial R_{j}(\frac{\pi}{2}),$$
it follows from $(\ref{R_theta})$ and $(\ref{60,5})$ that
\begin{equation}\label{60,6}
R_{\theta}^{-1}\frac{\partial R_{\theta}}{\partial \theta_j}= R_m^{-1}...R_{j+1}^{-1}  R_{j}(\frac{\pi}{2})  \Pi_j  R_{j+1}...R_m.
\end{equation}
By the same argument, we drive that $\frac{\partial R_{\theta}^{-1}}{\partial \theta_j}\partial R_{\theta}$ has the same expression, thus, (\ref{61,5}) holds from (\ref{A_i}) and (\ref{60,6}).
\bigskip

Now, we give the proof of Lemma \ref{prod}.\\
\noindent{\bf \it Proof of Lemma \ref{prod}}:\\
i) We first give the expression of the $m\times m$ matrix $R_\theta$ defined (\ref{R_theta}). Indeed, using (\ref{matrice}) and (\ref{R_theta}), we have:

\begin{eqnarray}\label{n150}
R_\theta\equiv \begin{pmatrix}
 \varphi_{2,m}&-\sin \theta_2&\cdots &-\sin \theta_j  \varphi_{2,l-1}&\cdots&-\sin \theta_m \varphi_{2,m-1}\\ \\
 
\sin\theta_2 \varphi_{3,m}&\cos \theta_2& & & &.\\

 \vdots& &\ddots& &  \text{\huge {R}}_{\Huge \theta, k,l}    &\vdots\\
 
\vdots     & & &\cos \theta_j& &\vdots\\
 
 \sin \theta_{k}  \varphi_{k+1,m} & & \text{\Huge {0}} & & &\vdots\\
\vdots & & &   &\ddots&\vdots\\ \\
 \sin\theta_m\varphi_{m+1,m}&0 & \cdots& \cdots&0&\cos \theta_m\\

\end{pmatrix}
\end{eqnarray}
where for $k\ge 1$, $l\ge 2$:
\begin{equation} R_{\theta, k, l}=\left\{
\begin{array}{l}
 -\sin \theta_l \varphi_{2,l-1}\;  \mbox{ if } k=1\\
-\sin \theta_k \sin \theta_l \varphi_{k+1,l-1}\; \mbox{ if } 2\le k\le l-1\\
\cos \theta_k\; \mbox{ if } k=l\\
\; 0\; \; \mbox{ if } k\ge l+1
\end{array}
\right . \label{n151}
\end{equation}
with
\begin{equation} \label{n151,5}
\varphi_{ k, l}=\left\{
\begin{array}{l}
\prod\limits_{n=k}^{l}\cos \theta_n \mbox{ if } k \le l\\
1\mbox{ if } k\ge l+1.
\end{array}
\right . 
\end{equation}

\bigskip

In fact, we will prove the following identities, which imply item $i)$:\\
$(A)$ For all $i$,$j$ $\in \{2,...,m\}$, such that $  i\neq j$, we have
$$<e_j,A_i (e_1)>=0.$$
$(B)$ For all $i\in \{2,...,m\}$

 $$< e_1,A_i e_1>=0.$$
$(C)$ For all $i\in \{2,...,m\}$, we have
\begin{eqnarray*}
< e_i,A_i e_1>=
 \varphi_{ i+1, m},
\end{eqnarray*}

where $A_i$ and $\varphi_{i+1,m}$ are given in (\ref{A_i}) and (\ref{n151,5}) .\\

\noindent$\blacktriangleright${\bf \it Proof of $(A)$}. Let $i$,$j$ $\in \{2,...,m\}$, such that $ i\neq j$. The idea is to compute $< R_\theta e_j,\frac{\partial R_\theta }{\partial \theta_i} e_1>$ instead of $<e_j,A_i e_1 >$. In fact, using the conservation of the inner product after a rotation and the fact that $A_i=R_{\theta}^{-1} \frac{\partial R_\theta}{\partial \theta_i}$ (by (\ref{A_i})), we have:
\begin{eqnarray}\label{n27}
<e_j,A_i e_1 >=<R_\theta e_j,R_\theta A_i e_1>=< R_\theta e_j,\frac{\partial R_\theta }{\partial \theta_i} e_1>.
\end{eqnarray}
In the following, we distinguish two cases:\\
- Case $1$: $i\le j-1$,\\
- Case $2$: $i\ge j+1$.\\
We first handle Case $1$.\\
{\bf Case $1$:} $i\le j-1$.
Using (\ref{n150}) and its derivative with respect to $\theta_i$, we write:
\begin{equation} R_\theta e_j=(R_\theta e_j)_{k=1,...,m}=\left\{
\begin{array}{l}
 -\sin \theta_j \varphi_{2,j-1}, \mbox{ if } k=1\\
-\sin \theta_k \sin \theta_j \varphi_{k+1,j-1}\; \mbox{ if } 2\le k\le j-1\\
\cos \theta_j\; \mbox{ if } k=j\\
\; 0\; \; \mbox{ if } k\ge j+1
\end{array}
\right . \label{n151,6}
\end{equation}
and
\begin{equation} \frac{\partial R_\theta }{\partial \theta_i} e_1= (\frac{\partial R_\theta }{\partial \theta_i} e_1)_{ k=1,...., m}=\left\{
\begin{array}{l}
-\sin \theta_i  \frac{\varphi_{2,m}}{\cos \theta_i}    \mbox{ if } k =1\\
-\sin \theta_i  \sin \theta_k \frac{\varphi_{k+1,m}}{\cos \theta_i} \mbox{ if } 2\le k\le i-1\\
\cos  \theta_i \varphi_{i+1,m}, \mbox{ if } k=i\\
0\mbox{ if } k\ge i+1.
\end{array}
\right . \label{n151,7}
\end{equation}
Therefore,
\begin{eqnarray}\label{n152}
 &&< R_\theta e_j,\frac{\partial R_\theta }{\partial \theta_i} e_1>=\sum_{k=1}^{m}R_{\theta,k,j}{\frac{\partial R_\theta }{\partial \theta_i}}_{k,1}\notag
\\&=&\sin \theta_i  \frac{\varphi_{2,m}}{\cos \theta_i} \sin \theta_j \varphi_{2,j-1}+\sum_{k=2}^{i-1}(\sin \theta_k \sin \theta_i  \frac{\varphi_{k+1,m}}{\cos \theta_i} \times \sin \theta_k \sin \theta_j  \varphi_{k+1,j-1})\notag\\
&-& \cos \theta_i \varphi_{i+1,m} \sin \theta_i \sin \theta_j \varphi_{i+1,j-1}\notag\\
&=& \sin \theta_i \sin \theta_j \left(   \frac{\varphi_{2,m}}{\cos \theta_i} \varphi_{2,j-1}+ \sum_{k=2}^{i-1}(\sin \theta_k^2  \frac{\varphi_{k+1,m}}{\cos \theta_i} \times \varphi_{k+1,j-1}) - \cos \theta_i  \varphi_{i+1,m} \varphi_{i+1,j-1}    \right).
\end{eqnarray}
In order to transform the sum term in the previous identity, we make in the following a finite induction where the parameter $q$ decreases from $i-1$ to $1$:
\begin{lem}\label{lemrec} We have:
\begin{align}  \label{nhr} 
\forall q \in \{1,...,i-1\},\;\sum_{k=2}^{i-1}\sin \theta_k^2  \frac{\varphi_{k+1,m}}{\cos \theta_i} &\times \varphi_{k+1,j-1}-\cos \theta_i \varphi_{i+1,m} \varphi_{i+1,j-1} =\\
\sum_{k=2}^{q}&\sin \theta_k^2  \frac{\varphi_{k+1,m}}{\cos \theta_i} \times \varphi_{k+1,j-1}-\frac{ \varphi_{q+1,m} \varphi_{q+1,j-1} }{\cos \theta_i}. \notag
\end{align}
\end{lem}
\noindent {\bf Remark:} If $q=1$, the sum in the right hand side is naturally zero.\\
\begin{proof} See below.
\end{proof}
Applying this Lemma, we conclude the proof of $(A)$ in Case $1$ (i. e. when $i\le j-1$). Indeed, from (\ref{n152}) and Lemma \ref{lemrec} with $q=1$ we write 
\begin{eqnarray*}\label{}
 < R_\theta e_j,\frac{\partial R_\theta }{\partial \theta_i} e_1>=\sin \theta_i \sin \theta_j \left(   \frac{\varphi_{2,m}}{\cos \theta_i}  \varphi_{2,j-1} - \frac{\varphi_{2,m}}{\cos \theta_i} \varphi_{2,j-1}\right)=0.
\end{eqnarray*}
It remains now to prove Lemma \ref{lemrec}.
\begin{proof}[Proof of Lemma \ref{lemrec}] First, we give the following:
\begin{cl} \label{lemrecc} We have
\begin{eqnarray*}
\varphi_{i,j-1}=\cos \theta_i \varphi_{i+1,j-1}.
\end{eqnarray*}
\end{cl}
\begin{proof} Since $i\le j-1$, we have two cases:\\
- If $i\le j-2$: trivial.\\
- If $i=j-1$:
 $\varphi_{i,j-1}=\varphi_{i,i} =\cos \theta_i  $ and $\varphi_{i+1,j-1}  = \varphi_{i+1,i}=1 $, and the result follows.
\end{proof}
Now, we are ready to start the proof of Lemma \ref{lemrec}. Let us prove the result using an induction with a decreasing index. For $q=i-1$, (\ref{nhr}) is satisfied using Claim \ref{lemrecc}. Assume now that (\ref{nhr}) is true for $q=i-1,...,2$ and let us prove it for $q-1$. Using (\ref{nhr}) with $q$, we write 
\begin{eqnarray*}
 &&\sum_{k=2}^{i-1}\sin \theta_k^2  \frac{\varphi_{k+1,m}}{\cos \theta_i} \times \varphi_{k+1,j-1} - \cos \theta_i  \varphi_{i+1,m} \varphi_{i+1,j-1}  =\\
&& \sum_{k=2}^{q-1}\sin \theta_k^2  \frac{\varphi_{k+1,m}}{\cos \theta_i} \times \varphi_{k+1,j-1}+ \sin \theta_q^2  \frac{\varphi_{q+1,m}}{\cos \theta_i} \times \varphi_{q+1,j-1} -\frac{ \varphi_{q+1,m} \varphi_{q+1,j-1} }{\cos \theta_i} =\\
&&  \sum_{k=2}^{q-1}\sin \theta_k^2  \frac{\varphi_{k+1,m}}{\cos \theta_i} \times \varphi_{k+1,j-1}- \cos \theta_q^2  \frac{\varphi_{q+1,m}}{\cos \theta_i} \times \varphi_{q+1,j-1}=  \\
 &&\sum_{k=2}^{q-1}\sin \theta_k^2  \frac{\varphi_{k+1,m}}{\cos \theta_i} \times \varphi_{k+1,j-1} -\frac{ \varphi_{q,m} \varphi_{q,j-1} }{\cos \theta_i}.
\end{eqnarray*}
Thus, (\ref{nhr}) is satisfied for $q-1$. This concludes the proof of Lemma \ref{lemrec} and identity $(A)$ when $i\le j-1$.
\end{proof}
Now, we handle Case $2$.\\
- {\bf Case} $2$: $i\ge j+1$.\\ Using (\ref{n151,6}) and (\ref{n151,7}), we write:
\begin{eqnarray}\label{n1521}
 &&< R_\theta e_j,\frac{\partial R_\theta }{\partial \theta_i} e_1>=\sum_{k=1}^{m}R_{\theta,k,j}{\frac{\partial R_\theta }{\partial \theta_i}}_{k,1}\notag
\\
&=& \sin \theta_i \sin \theta_j \left(   \frac{\varphi_{2,m}}{\cos \theta_i} \varphi_{2,j-1}+ \sum_{k=2}^{j-1}(\sin^2 \theta_k  \frac{\varphi_{k+1,m}}{\cos \theta_i} \times \varphi_{k+1,j-1}) - \cos \theta_j \frac{ \varphi_{j+1,m}}{\cos \theta_i}     \right).
\end{eqnarray}
In order to transform the sum term in the previous identity, we make in the following a finite induction where the parameter $q$ decreases from $j-1$ to $1$:
\begin{lem}\label{lemrec1} We have:
\begin{align}  \label{nhr1} 
\forall q \in \{1,...,j-1\},\;\sum_{k=2}^{j-1}\sin^2 \theta_k  \frac{\varphi_{k+1,m}}{\cos \theta_i} &\times \varphi_{k+1,j-1}-\cos \theta_j  \frac{ \varphi_{j+1,m}}{\cos \theta_i}   =\\
\sum_{k=2}^{q}&\sin^2 \theta_k \frac{\varphi_{k+1,m}}{\cos \theta_i} \times \varphi_{k+1,j-1}-\varphi_{q+1,j-1} \frac{ \varphi_{q+1,m}}{\cos \theta_i}  . \notag
\end{align}
\end{lem}
\noindent {\bf Remark:} If $q=1$, the sum in the right hand side is naturally zero.\\
\begin{proof} See below.
\end{proof}
Applying this Lemma, we conclude the proof of $(A)$ in Case $2$ (i. e. when $i\ge j+1$). Indeed, from (\ref{n1521}) and Lemma \ref{lemrec1} with $q=1$ we write 
\begin{eqnarray*}\label{}
 < R_\theta e_j,\frac{\partial R_\theta }{\partial \theta_i} e_1>=\sin \theta_i \sin \theta_j \left(   \frac{\varphi_{2,m}}{\cos \theta_i}  \varphi_{2,j-1} - \frac{\varphi_{2,m}}{\cos \theta_i} \varphi_{2,j-1}\right)=0.
\end{eqnarray*}
It remains now to prove Lemma \ref{lemrec1}.
\begin{proof}[Proof of Lemma \ref{lemrec1}] We prove the result using an induction with a decreasing index. For $q=j-1$, (\ref{nhr1}) is satisfied. Assume now that (\ref{nhr1}) is true for $q=j-1,...,2$ and let us prove it for $q-1$. Using (\ref{nhr1}) with $q$, we write 
\begin{eqnarray*}
&&\sum_{k=2}^{j-1}\sin^2 \theta_k  \frac{\varphi_{k+1,m}}{\cos \theta_i} \times \varphi_{k+1,j-1}-\cos \theta_j  \frac{ \varphi_{j+1,m}}{\cos \theta_i}   =\\
&&\sum_{k=2}^{q}\sin^2 \theta_k \frac{\varphi_{k+1,m}}{\cos \theta_i} \times \varphi_{k+1,j-1}-\varphi_{q+1,j-1} \frac{ \varphi_{q+1,m}}{\cos \theta_i}= \\
&&\sum_{k=2}^{q-1}\sin^2 \theta_k \frac{\varphi_{k+1,m}}{\cos \theta_i} \times \varphi_{k+1,j-1}+\sin^2 \theta_q \frac{\varphi_{q+1,m}}{\cos \theta_i} \times \varphi_{q+1,j-1}-\varphi_{q+1,j-1} \frac{ \varphi_{q+1,m}}{\cos \theta_i} =\\
&&\sum_{k=2}^{q-1}\sin^2 \theta_k \frac{\varphi_{k+1,m}}{\cos \theta_i} \times \varphi_{k+1,j-1}- \frac{ \varphi_{q,m}}{\cos \theta_i}\varphi_{q,j-1}.
\end{eqnarray*}
Thus, (\ref{nhr1}) is satisfied for $q-1$. This concludes the proof of Lemma \ref{lemrec1}.
\end{proof}

\noindent$\blacktriangleright${\bf \it Proof of $(B)$}: As for (\ref{n27}) we have:
$$<e_1,A_i e_1 >=<R_\theta e_1,  \frac{\partial R_\theta }{\partial \theta_i} e_1>.$$
From (\ref{n150}), we have:
$$R_\theta e_1=( \varphi_{2,m},\sin\theta_2\varphi_{3,m}\,\cdots, \sin \theta_{i}  \varphi_{i+1,m} ,\cdots,\sin\theta_m).$$
Therefore, using  (\ref{n151,7}), we have:
\begin{eqnarray*}\label{}
 < R_\theta e_1,\frac{\partial R_\theta }{\partial \theta_i} e_1>=-\cos \theta_i\sin\theta_i\left((\frac{\displaystyle\varphi_{2,m}}{\cos \theta_i})^2  +\sum_{k=2}^{i-1}\sin^2\theta_k (\frac{\displaystyle\varphi_{k+1,m}}{\cos \theta_i})^2 \right)+\cos \theta_i\sin\theta_i     (\varphi_{i+1,m})^2.
\end{eqnarray*}

In order to transform the sum term in the previous identity, we make in the following a finite induction:

\begin{lem}\label{lemrec2} We have:
\begin{eqnarray}\label{eqrec}
\forall q\in \{2,..., i-1\},\;\displaystyle(\frac{\displaystyle\varphi_{2,m}}{\cos \theta_i})^2  +\sum_{l=2}^{i-1}\sin^2\theta_l (\frac{\displaystyle\varphi_{l+1,m}}{\cos \theta_i})^2=(\frac{\displaystyle\varphi_{q,m}}{\cos \theta_i})^2  +\sum_{l=q}^{i-1}\sin^2\theta_l (\frac{\displaystyle\varphi_{l+1,m}}{\cos \theta_i})^2.
\end{eqnarray}
\end{lem}
\noindent {\bf Remark:} If $q=i$, the sum in the right hand side is naturally zero.\\
Using Lemma \ref{lemrec2} with $q=i$ we get
\begin{eqnarray*}
 < R_\theta e_1,\frac{\partial R_\theta }{\partial \theta_i} e_1>=-\cos \theta_i\sin\theta_i   (\varphi_{i+1,m})^2 +\cos \theta_i\sin\theta_i    (\varphi_{i+1,m})^2=0,
 \end{eqnarray*}
which yields the result. In order to conclude $(B)$ we give the proof of Lemma \ref{lemrec2}.
\begin{proof}[Proof of Lemma \ref{lemrec2}] We proceed by induction for $q\in \{2,..., i-1\}$. For $q=2$, (\ref{eqrec}) is satisfied. Assume that (\ref{eqrec}) is true for $q=2,...,i-1$ and prove it for $q+1$. Using (\ref{eqrec}) with $q$, we write
\begin{eqnarray*}\label{}
\displaystyle(\frac{\displaystyle\varphi_{2,m}}{\cos \theta_i})^2  &+&\sum_{l=2}^{i-1}\sin^2\theta_l (\frac{\displaystyle\varphi_{l+1,m}}{\cos \theta_i})^2
=(\frac{\displaystyle\varphi_{q,m}}{\cos \theta_i})^2  +\sum_{l=q}^{i-1}\sin^2\theta_l (\frac{\displaystyle\varphi_{l+1,m}}{\cos \theta_i})^2
\\&=&\cos^2\theta_q (\frac{\displaystyle\varphi_{q+1,m}}{\cos \theta_i})^2  +\sin^2\theta_q (\frac{\displaystyle\varphi_{q+1,m}}{\cos \theta_i})^2+\sum_{l=q+1}^{i-1}\sin^2\theta_l (\frac{\displaystyle\varphi_{l+1,m}}{\cos \theta_i})^2\\&=&(\frac{\displaystyle\varphi_{q+1,m}}{\cos \theta_i})^2  +\sum_{l=q+1}^{i-1}\sin^2\theta_l (\frac{\displaystyle\varphi_{l+1,m}}{\cos \theta_i})^2.
\end{eqnarray*}
Thus (\ref{eqrec}) is satisfied for $q+1$. This concludes the proof of Lemma \ref{lemrec2} and identity $(B)$.
\end{proof}
\bigskip

\noindent$\blacktriangleright${\bf \it Proof of $(C)$}: Consider $i\in \{2,...,m\}$. As for (\ref{n27}) we have:
$$<e_i,A_i e_1 >=<R_\theta e_i,  \frac{\partial R_\theta }{\partial \theta_i} e_1>.$$
Using (\ref{n151}) and (\ref{n151,7}) 
\begin{eqnarray}\label{n158}
<e_i,A_i e_1 >=\sin^2\theta_i   \varphi_{i+1,m}  \left(  \varphi_{2,i-1}^2+\sum_{k=2}^{i-1}\sin^2\theta_k \varphi_{k+1,i-1}^2
\right)+\cos^2 \theta_i  \varphi_{i+1,m}.
\end{eqnarray}

In order to transform the sum term in the previous identity, we make in the following a finite induction:

\begin{lem}\label{lemrec3} We have: $\forall q\in \{2,..., i\}$,
\begin{eqnarray}\label{equarec}
\varphi_{2,i-1}^2+\sum_{l=2}^{i-1}\sin^2\theta_l \varphi_{l+1,i-1}^2=\varphi_{q,i-1}^2+\sum_{l=q}^{i-1}\sin^2\theta_l \varphi_{l+1,i-1}^2.
\end{eqnarray}
\end{lem}
\noindent {\bf Remark:} If $q=i$, the sum in the right hand side is naturally zero.\\
From (\ref{n158}) and (\ref{equarec}) with $q=i$ we get
\begin{eqnarray*}
 < R_\theta e_i,\frac{\partial R_\theta }{\partial \theta_i} e_1>=\sin^2\theta_i  \varphi_{i+1,m}+\cos^2\theta_i  \varphi_{i+1,m}=  \varphi_{i+1,m}.
\end{eqnarray*}
which yields the result. In order to conclude $(C)$ we give the proof of Lemma \ref{lemrec3}.
\begin{proof}[Proof of Lemma \ref{lemrec3}] We proceed by induction for $q\in \{2,..., i\}$. For $q=2$, (\ref{equarec}) is satisfied. Assume now that (\ref{equarec}) is true for $q=2,...,i-1$ and prove it for $q+1$. Using (\ref{equarec}) with $q$, we write

\begin{eqnarray*}\label{}
\varphi_{2,i-1}^2+\sum_{l=2}^{i-1}\sin^2\theta_l \varphi_{l+1,i-1}^2&=&\varphi_{q,i-1}^2+\sum_{l=q}^{i-1}\sin^2\theta_l \varphi_{l+1,i-1}^2\\
&=&\cos^2\theta_q \varphi_{q+1,i-1}^2+\sin^2\theta_q \varphi_{q+1,i-1}^2+\sum_{l=q+1}^{i-1}\sin^2\theta_l \varphi_{l+1,i-1}^2\\
&=&\varphi_{q+1,i-1}^2+\sum_{l=q+1}^{i-1}\sin^2\theta_l \varphi_{l+1,i-1}^2.
\end{eqnarray*}

Thus (\ref{equarec}) is satisfied for $q+1$. This concludes the proof of Lemma \ref{lemrec3}.
\end{proof}

ii)We recall from (\ref{R_theta}) that we have
$$\frac{\partial R_\theta}{\partial \theta_j}=R_2 \cdots R_{j-1} \frac{\partial R_j}{\partial \theta_j} R_{j+1} \cdots R_m,$$
so by (\ref{A_i}), $A_j$ is given explicitly by
$$A_j=R_m^{-1} R_{m-1}^{-1} \cdots R_{j}^{-1} \frac{\partial R_j}{\partial \theta_j} R_{j+1} \cdots R_m.$$
From a straightforward geometrical observation, we can see that the rotation conserves the euclidien norm in $\mathbb{R}^m$. For $\frac{\partial R_j}{\partial \theta_j} $, it can be seen as a composition of a projection on the plane $(e_1,e_j)$ and a rotation with angle $\theta_j+\frac{\pi}{2}$, which decreases the norm. This concludes the proof of Lemma \ref{prod}.

\noindent{\bf Address:}\\
Universit\'e de Cergy-Pontoise, 
Laboratoire Analyse G\'eometrie Mod\'elisation, \\
  CNRS-UMR 8088, 2 avenue Adolphe Chauvin
95302, Cergy-Pontoise, France.
\\ \texttt{e-mail: asma.azaiez@u-cergy.fr}\\
Universit\'e Paris 13, Institut Galil\'ee, Laboratoire Analyse G\'eometrie et Applications, \\
  CNRS-UMR 7539, 99 avenue J.B. Cl\'ement 93430, Villetaneuse, France.
\\ \texttt{e-mail: zaag@math.univ-paris13.fr}
\end{document}